\documentclass{amsart}
\usepackage{amsmath,amssymb,amsfonts,amsthm}
\usepackage{color}

\usepackage{graphicx}
\allowdisplaybreaks

\theoremstyle{plain}
\newtheorem{theorem}{Theorem}[section]
\newtheorem{proposition}[theorem]{Proposition}
\newtheorem{lemma}[theorem]{Lemma}
\newtheorem{corollary}[theorem]{Corollary}
\newtheorem{question}[theorem]{Question}
\newtheorem{definition}[theorem]{Definition}
\newtheorem{claim}[theorem]{Claim}
\newtheorem{remark}[theorem]{Remark}

\theoremstyle{definition}
\theoremstyle{remark}

\begin{document}

\title
[On strong spaceability of continuous functions]{On strong spaceability of continuous functions and fractal dimensions}
\author{Jia Liu${^1}$  Saisai Shi${^2}$ \and Zhenliang Zhang$^3$}
\address{$^{1}$ Institute of Statistics and Applied Mathematics, Anhui University of Finance and Economics, 233030, Bengbu, P. R. China}
\email{liujia860319@163.com}
\address{${^2}$
Institute of Statistics and Applied Mathematics, Anhui University of Finance and Economics, 233030, Bengbu, P. R. China}
\email{saisai\_shi@126.com}
\address{${^3}$
School of Mathematical Sciences,
Chongqing Normal University, 401331, P. R. China}
\email{zhliang\_zhang@163.com}

\thanks{This work was supported by NSFC No. 12501109  and Scientific Research Project of Colleges and Universities in Anhui Province (2023AH050264).\\}
\date{}
\begin{center}
\end{center}
\begin{abstract}
Given $s\in(1,2]$,  define
$$H_s[0,1]=\{f\in C[0,1]:{\dim}_HG_f([0,1])=s\}$$
and
$$\overline{B}_s[0,1]=\{f\in C[0,1]:\overline{{\dim}}_BG_f([0,1])=s\}.$$
The main goal of this paper is to study the $(\alpha,\beta)$-lineability/spaceability of the sets $H_s[0,1]$ and $\overline{B}_s[0,1]$. As a principal result, we prove that $H_s[0,1]$ is $(p,\mathfrak{c})$-spaceable for $p=1,2$ and $(n,n+m)$-lineable for any $m,n\in\mathbb{N}$. This partially answers a question raised by Liu et al. concerning the Hausdorff dimension of graphs of continuous functions.

Furthermore, for a cardinal number $\alpha$, we prove that $\overline{B}_s[0,1]$ is $(\alpha,\mathfrak{c})$-spaceable if and only if $\alpha<\aleph_0$. This completely resolves an open question raised by Liu et al. concerning the upper box dimension of graphs of continuous functions.

\medskip

\smallskip


\end{abstract}

\maketitle

\section{Introduction}

Let us recall the relevant notations and symbols.  As usual, $\mathbb{N}$ denotes the set of all positive integers and $\mathbb{K}$ represents the real scalar field $\mathbb{R}$ or the complex scalar field $\mathbb{C}$. Let $\aleph_0$ and $\mathfrak{c}$ denote the cardinalities of $\mathbb{N}$ and $\mathbb{R}$, respectively. We endow the space $C[0,1]$ of real-valued continuous functions $[0,1] \to \mathbb{R}$ with the maximum norm. Observe that $C[0,1]$ is an infinite dimensional separable Banach space, so its dimension is exactly $\mathfrak{c}$. Given a function $f: X \subseteq \mathbb{R} \to \mathbb{R}$, we denote its graph by $G_f(X)$, i.e.,
\[
G_f(X) = \{(x, f(x)) : x \in X\} \subseteq X \times \mathbb{R}.
\]

\medskip
In the present paper, we investigate the algebraic genericity of continuous functions in terms of the Hausdorff dimension and upper box dimension. The concepts of lineability and spaceability were originally coined by V. I. Gurariy and they first appeared in \cite{ AGS,GQ,Sea}, whose original goal was to find large algebraic structures within sets of special objects. In \cite{FPT}, the stronger notions $(\alpha,\beta)$-lineability and $(\alpha,\beta)$-spaceability were introduced in an attempt to investigate how far positive results for lineability and spaceability remain valid under stronger hypotheses. Let us recall some specific definitions.

Let $X$ be a vector space and let $\alpha$, $\beta$ be two cardinal numbers with $\alpha<\beta$. We say that a subset $A\subseteq X$ is
\begin{itemize}
  \item lineable if there is an infinite dimensional subspace $M$ such that $M\subseteq A\cup\{0\}$.
  \item $\alpha$-lineable if there is an $\alpha$-dimensional subspace $M$ such that $M\subseteq A\cup\{0\}$.
  \item maximal-lineable if it is ${\dim}(X)$-lineable.
    \item $(\alpha,\beta)$-lineable if it is $\alpha$-lineable and for each $\alpha$-dimensional subspace $W_\alpha \subseteq A \cup \{0\}$ there is a $\beta$-dimensional subspace $W_\beta$ such that

   \begin{equation}\label{eqdef}
    W_\alpha \subseteq W_\beta \subseteq A \cup \{0\}.
   \end{equation}
\end{itemize}

\medskip
If, in addition, $X$ is a topological vector space, then the subset $A$ is said to be

\begin{itemize}
  \item  spaceable  whenever there is a closed infinite dimensional subspace $M$
      such that $M\subseteq A\cup\{0\}$.
  \item dense-lineable  whenever there is an infinite dimensional dense  subspace $M$ such that $M\subseteq A\cup\{0\}$.
  \item $\alpha$-dense-lineable whenever there is an $\alpha$-dimensional dense subspace $M$ such that $M\subseteq A\cup\{0\}$.
  \item  maximal-dense-lineable whenever there is a ${\dim}(X)$-dimensional dense subspace $M$ such that $M\subseteq A\cup\{0\}$.
  \item $(\alpha,\beta)$-spaceable if the subspace $W_\beta$ satisfying (\ref{eqdef}) can always be chosen closed.
  \item $(\alpha,\beta)$-dense-lineable if it is $\alpha$-lineable and for each $\alpha$-dimensional subspace $W_\alpha \subseteq A \cup \{0\}$ there is a $\beta$-dimensional dense subspace $W_\beta$ such that
    \[
        W_\alpha \subseteq W_\beta \subseteq A \cup \{0\}.
    \]
\end{itemize}

And, the term algebrability is defined in a similar approach, which was introduced in \cite{AS}. If $X$ is a vector space contained in some (linear) algebra, then $A$ is called:
\begin{itemize}
  \item   algebrable if there is a subalgebra $C\subseteq X$ so that $C\subseteq A\cup\{0\}$ and $C$
has an infinite minimal system of generators (here, by $S=\{z_{s}\}$ is a minimal set of generators of $C$, we mean that $C = \mathcal{A}(S)$ is the algebra generated by $S$, and for every $s_0, z_{s_0}\notin\mathcal{A}(S\setminus \{z_{s_0} \})$.
 \item dense-algebrable if, in addition, $C$ can be taken dense in $X$.
\end{itemize}

\medskip
Recently, the notions of pointwise $\alpha$-lineability (respectively, spaceability) in \cite{PellRa} were introduced as follows:
\begin{itemize}
    \item a subset $A$ of a vector space $X$ is called pointwise $\alpha$-lineable if, for each $x \in A$, there is an $\alpha$-dimensional subspace $W_x$ such that
\begin{equation}\label{eqdef2}
        x \in W_x \subseteq A \cup \{0\}.
\end{equation}

  \item if $X$ is a topological vector space and the subspace $W_x$ satisfying (\ref{eqdef2}) can always be chosen closed (dense), we say that $A$ is respectively pointwise $\alpha$-spaceable (pointwise $\alpha$-dense-lineable). If $\alpha = \dim(X)$, we say that $A$ is respectively maximal pointwise spaceable (maximal pointwise dense-lineable).
 \end{itemize}

 It is evident that $(\alpha,\beta)$-lineability (resp. spaceability) implies 
$\beta$-lineability (resp. spaceability) and that pointwise $\beta$-lineability (resp. spaceability) implies  $(1,\beta)$-lineability (resp. spaceability). In general, the converse is not true (see Example 2.2 in \cite{PellRa}).
\smallskip

Since the concepts of lineability and spaceability emerged in the 2000s, the research on linear structures within exotic settings has been attracting extensive attention from numerous scholars
and these two properties bear vital practical values in a wide range of disciplines including
Real and Complex Analysis \cite{Alb,APG,AGS,GC,EGS,Sea}, Series Theory \cite{AES,Bay}, Operator Theory \cite{BDP,BF,PT}, Classical Banach Sequence Space Theory \cite{BBFP,BDFP,GC,CarS}, Measure and Probability Theory \cite{BFG,CFMS,FST,MPPS}, Set Theory \cite{GS}, ODE \cite{BBFP} as well as Fractal Geometry \cite{BFBS,CFS,LZS} and Multi-Fractal Analysis \cite{EJ}.
During the last three decades, many authors have devoted much effort to studying pathological real-valued functions \cite{AGPS,APG,AGS,AS,GFVS,GPS,BFBS,CS1,CS2,EGS,FGK,MFSS,GFS}.
One of the earliest breakthroughs in this direction was due to Gurariy, who proved that the set of continuous nowhere differentiable functions on $[0,1]$ is lineable \cite{Gu}.
Subsequently, Fonf, Gurariy and Kadets \cite{FGK} proved that the set of continuous nowhere differentiable functions on $[0, 1]$ is spaceable in $C[0, 1]$.
Enflo et al. \cite{EGS} demonstrated that for every infinite dimensional closed subspace $X$ of $C[0, 1]$, the set of functions in $X$ having infinitely many zeros in $[0, 1]$ is spaceable in $X$.
A vast literature on this topic has been built during the last two decades, interested readers may refer to the survey paper \cite{GPS} or the monograph \cite{AGPS} for a more detailed and comprehensive treatment.

\medskip
This paper contributes to the study of large vector spaces of functions endowed with specific fractal properties.
To our knowledge, N. Albuquerque et al. \cite{AlGPS} first considered the lineability and algebrability of Peano curves, followed by D. Cariello et al. \cite{CFS}, who studied the algebrability of self-similar functions.
Later, A. Bonilla et al. \cite{BFBS} explored the algebraic genericity of continuous functions whose graphs have prescribed fractal dimensions. They proved that given $s\in(1,2]$, the set of functions whose graphs have Hausdorff dimension $s$ (i.e., ${\dim}_HG_f([0,1])=s$) is spaceable and the set of functions whose graphs have both Hausdorff and  box dimensions $s$ is maximal-dense-lineable.

Very recently, by addressing two open questions raised by Bonilla et al. \cite[Question 2.1 and Question 3.4]{BFBS}, Liu et al. \cite{LZS} proved that
given $s\in(1,2]$, the set of functions $f\in C[0,1]$ whose graphs have both box and Hausdorff dimensions $s$ everywhere in $[0,1]$
is maximal-dense-lineable and also dense-algebrable.  They also put forward the following conjecture \cite[Question 5]{LZS}.
Notably, C. Esser et al. \cite{EMVVS} recently proved that the set of functions $f\in C[0,1]$ with ${\dim}_HG_f([0,1])=s\in(1,2]$ is strongly $\mathfrak{c}$-algebrable, a notion originally introduced in \cite{BG}.

\begin{question}\label{qu1} Let ${\dim}$ be one of ${\dim}_H$, $\overline{{\dim}}_B$ and $\underline{{\dim}}_B$. Let $\alpha, \beta$ be two cardinal numbers with $1<\alpha<\beta$. Given $s\in(1,2]$, is it possible to obtian the pointwise spaceability or $(\alpha,\beta)$-lineability/spaceability of the set of functions $f\in C[0,1]$ with ${\dim}G_f([0,1])=s$?
\end{question}
In this paper, we intend to answer the above question. The paper is organized as follows.
In forthcoming Section 2, we shall provide the definitions of Hausdorff and box dimensions and a series of auxiliary results which will be used throughout the paper.
In Section 3, we obtain a partial answer to Question \ref{qu1} by showing that the set of functions $f\in C[0,1]$ with ${\dim}_HG_f([0,1])=s\in (1,2]$ is $(p,\mathfrak{c})$-spaceable for $p=1,2$ and $(n,n+m)$-lineable for any $n,m\in\mathbb{N}$.
In Section 4, we prove that the set of functions $f\in C[0,1]$ with $\overline{{\dim}}_BG_f([0,1])=s\in (1,2]$ is $(\alpha,\mathfrak{c})$-spaceable if and only if the cardinal number $\alpha<\aleph_0$. This yields a complete solution to Question \ref{qu1} regarding upper box dimension. At the end of this section, we briefly discuss the $(\alpha,\beta)$-lineability of the set of functions whose graphs have both Hausdorff and box dimensions $s\in(1,2]$.

\section{Preliminaries}
In this section, for the convenience of the reader, we recall the definitions of Hausdorff dimension and box dimension, and then provide some useful lemmas which will be needed throughout this paper.

For a non-empty set $A \subseteq \mathbb{R}^d$, and any $\delta > 0$, let
\[
\mathcal{H}^s_\delta(A) = \inf\left\{ \sum_{i=1}^\infty |U_i|^s : A \subseteq \bigcup_{i=1}^\infty U_i,\ |U_i| \leq \delta,\ \forall i \right\}
\]
where $|U|$ represents the diameter of the set $U$, and let
\[
\mathcal{H}^s(A) = \lim_{\delta \to 0} \mathcal{H}^s_\delta(A).
\]
The Hausdorff dimension of $A$ is
\[
\dim_H A = \inf\{ s \geq 0 : \mathcal{H}^s(A) = 0 \} = \sup\{ s \geq 0 : \mathcal{H}^s(A) = \infty \}.
\]

\medskip
Furthermore, given a non-empty and bounded subset $A$ of $\mathbb{R}^d$, let $N_\delta(A)$ denote the smallest number of sets of diameter at most $\delta$  needed to cover $A$. Then the lower and upper box dimensions of $A$ are
\[
\underline{\dim}_B A = \liminf_{\delta \to 0} \frac{\log N_\delta(A)}{-\log \delta} \quad \text{and} \quad \overline{\dim}_B A = \limsup_{\delta \to 0} \frac{\log N_\delta(A)}{-\log \delta}.
\]
If these two dimensions coincide, we call this common value the box dimension of $A$, denoted by
\[
\dim_B A = \lim_{\delta \to 0} \frac{\log N_\delta(A)}{-\log \delta}.
\]
For more details, we refer the reader to \cite{Fal}. Notice that
\[
\dim_H A \leq \underline{\dim}_B A \leq \overline{\dim}_B A
\]
and
\begin{equation}\label{equnew1}
\dim(A\times [0,1])= {\dim}A+1
\end{equation}
for any non-empty bounded subset $A$ of $\mathbb{R}^d$, where ${\dim}$ denotes any one of ${\dim}_H$, $\underline{\dim}_B$ and $\overline{\dim}_B$.

\medskip
\begin{definition}
Let $K \subseteq [0,1]$ and $f \in C(K)$.
If for all but countably many points $x \in K$, there exists some $r_x > 0$ such that $f$ is Lipschitz on $K \cap B(x,r_x)$, then $f$ is said to be \emph{nearly locally Lipschitz on $K$}. Here and in what follows, let $B(x,r)$ denote the open ball centered at $x$ with radius $r$.
\end{definition}

\medskip
By the countable stability and bi-Lipschitz invariance of the Hausdorff dimension, we can easily obtain the following lemma.
\begin{lemma}\label{lemL}
Let $K \subseteq \mathbb{R}$ , $f, g \in C(K)$. If $f$ is nearly locally Lipschitz on $K$, then we have
\[
\dim_H G_{f+g}(K) = \dim_H G_g(K).
\]
In particular, the above equality holds for a Lipschitz function $f$.
\end{lemma}
By the bi-Lipschitz invariance of lower box and upper box dimensions, we also have the following lemma.
\begin{lemma}\label{lemx2}
Let $K\subseteq \mathbb{R}$ be a bounded set and let $f, g \in C(K)$ be two bounded functions. If $g$ is a Lipschitz function, then
\[
\dim G_{f+g}(K) = \dim G_f(K),
\]
where $\dim$ denotes any one of $\underline{\dim}_B$ and $\overline{\dim}_B$.
\end{lemma}

\medskip
\begin{lemma}\cite{LW} \label{lemLW}
Assume that $K\subseteq [0,1]$ is a compact set and $f \in C(K)$. Let $\widetilde{f}$ be the linear extension of $f$ to the whole interval $[0,1]$, i.e.,
\[
\widetilde{f}(x)=
\begin{cases}
f(x), & x \in K, \\
\text{affine}, & \text{on each component of } [0,1]\setminus K.
\end{cases}
\]
Then
\[
\dim_H G_{\widetilde{f}}([0,1]) = \max\left\{ \dim_H G_f(K), 1 \right\}.
\]
\end{lemma}
\medskip
\begin{lemma}\cite{FF}\label{lemfal-fr}
Let $f, g \in C[0,1]$. Then
\[
\overline{\dim}_B G_{f+g}([0,1])\leq \max\left\{ \overline{\dim}_B G_f([0,1]), \overline{\dim}_B G_g([0,1]) \right\}.
\]
\end{lemma}

\medskip
\begin{lemma}\cite{Fal}\cite{FF}\label{lemfalc}
Let $f, g \in C[0,1]$. If $\overline{\dim}_B G_f([0,1]) \neq \overline{\dim}_B G_g([0,1])$, then
\[
\overline{\dim}_B G_{f+g}([0,1])= \max\left\{ \overline{\dim}_B G_f([0,1]), \overline{\dim}_B G_g([0,1]) \right\}.
\]
\end{lemma}

\medskip
\begin{lemma}\cite{LL}\label{lemLL}
Let $K \subseteq [0,1]$ be a compact set. Then for any $f \in C(K)$, we have
\[
\overline{\dim}_BG_{\widetilde{f}}([0,1]) \leq \overline{\dim}_B K + 1
\quad \text{and} \quad
\underline{\dim}_B G_{\widetilde{f}}([0,1]) \leq \underline{\dim}_B K + 1,
\]
where $\widetilde{f} \in C[0,1]$ is the linear extension of $f$ to the whole interval $[0,1]$.
\end{lemma}
\medskip
For convenience the reader, we provide a proof of the lemma.
\begin{proof}
Write $O_f(A)$ for the oscillation of a bounded function $f$ on $A$, i.e.,
\[
O_f(A) = \sup_{x,y \in A} |f(x)-f(y)|.
\]
Then by the definition of $\widetilde{f}$, we have $O_{\widetilde{f}}[0,1] = O_f(K)$. For each $n \in \mathbb{N}^+$, denote by $\Omega_n(K)$ the collection of all closed dyadic intervals of order $n$,
which intersect $K$, i.e.,
\[
\Omega_n(K) = \left\{ [i 2^{-n}, (i+1) 2^{-n}] : [i 2^{-n}, (i+1) 2^{-n}] \cap K \neq \emptyset,\ i \in \mathbb{Z} \right\}.
\]
Then the number of components of
\[
[0,1] \setminus \bigcup_{I \in \Omega_n(K)} I
\]
is at most $\#\Omega_n(K)+1$. It follows from the definition of $\widetilde{f}$ that it is affine on each component of $[0,1] \setminus \bigcup_{I \in \Omega_n(K)} I$.
This implies that for each such component $J$, the number of $2^{-n}$-mesh squares which intersect $G_{\widetilde{f}}(J)$, denoted by $N_{2^{-n}}(\widetilde{f}|_J)$, is less than $3(2^n O_f + 2)$. Then we get that
\begin{align*}
N_{2^{-n}}(\widetilde{f}) &\leq 3(2^n O_f + 2)(\#\Omega_n(K)+1) + (2^n O_f + 2) \#\Omega_n(K) \\
&\leq C 2^n \#\Omega_n(K),
\end{align*}
where $C>0$ is a constant independent of $n$.

It follows that
\[
\overline{\dim}_B G_{\widetilde{f}}([0,1]) = \limsup_{n \to \infty} \frac{\log N_{2^{-n}}(\widetilde{f})}{-\log 2^{-n}} \leq \overline{\dim}_B K + 1,
\]
and
\[
\underline{\dim}_B G_{\widetilde{f}}([0,1]) = \liminf_{n \to \infty} \frac{\log N_{2^{-n}}(\widetilde{f})}{-\log 2^{-n}} \leq \underline{\dim}_B K + 1.
\]
\end{proof}
\begin{lemma}\cite{BDE,BH}\label{balk1}
Let $K\subseteq\mathbb{R}^d$ be an uncountable compact set. Then for a prevalent $f\in C(K)$, we have
\[
{\dim}_HG_f(K)={\dim}_HK+1.
\]
\end{lemma}
\begin{lemma}\cite{Bal,LTW}\label{balk2}
Let $K\subseteq\mathbb{R}^d$ be an uncountable compact set with at most finitely many isolated points. Then for a prevalent $f\in C(K)$, we have
\[
\overline{{\dim}}_BG_f(K)=\overline{{\dim}}_BK+1.
\]

\end{lemma}
\medskip
Recall that an $F$-space is a topological space whose topology can be defined by a complete translation-invariant metric. Note that the continuous function space $C(K)$ for a compact subset $K$ of $\mathbb{R}$ is an $F$-space when equipped with maximum norm. We also recall that (see \cite{AGPS}) if $A, B$ are subsets of a vector space, then $A$ is called stronger than $B$ whenever $A+B\subseteq A$.  The following well-known result is due to F\'{a}varo et al. \cite{FPRR}.
\begin{lemma}\cite{FPRR}\label{lemF1}
Let $\alpha\geq \aleph_0$ and $X$ be an $F$-space. Let $A, B$ be subsets of $X$ such that $A$ is $\alpha$-lineable and $B$ is $1$-lineable. If $A\cap B=\emptyset$ and $A$ is stronger than $B$, then $A$ is not $(\alpha, \beta)$-spaceable, regardless of the cardinal number $\beta$.
\end{lemma}
\medskip
Assume that $s \in (0,+\infty)$ and $K \subseteq [0,1]$. We denote $H_s(K)$, $H_{<s}(K)$, $\overline{B}_s(K)$ and $\overline{B}_{<s}(K)$ by
\noindent
\[
H_s(K) = \left\{ f \in C(K) : {\dim}_H G_f(K) = s \right\},
\]
\[
H_{<s}(K) = \left\{ f \in C(K) : {\dim}_H G_f(K) < s \right\},
\]
\[
\overline{B}_s(K) = \left\{ f \in C(K) : \overline{\dim}_B G_f(K) = s \right\},
\]
\[
\overline{B}_{<s}(K) = \left\{ f \in C(K) : \overline{\dim}_B G_f(K) < s \right\}.
\]
By replacing the corresponding dimensions in the above definitions, we can define $\underline{B}_s(K)$, $\underline{B}_{<s}(K)$ (regarding lower box dimension), $B_s(K)$ and $B_{<s}(K)$ (regarding box dimension) analogously. Note that if the functions in the above definitions involve upper box, lower box or box dimensions, we always assume them to be bounded functions.

\medskip
An immediate corollary of Lemma \ref{lemF1} is as follows.
\begin{corollary}\label{corn1}
Let $s\in(1,2]$ and let $\alpha$ be a cardinal number with $\alpha\geq\aleph_0$. Then none of $H_s[0,1]$, $\overline{B}_s[0,1]$, $\underline{B}_s[0,1]$ and $B_s[0,1]$ is $(\alpha,\beta)$-spaceable, regardless of the cardinal number $\beta$.
\end{corollary}
\begin{proof}
Let $A$ be any one of the sets involved in the corollary and  $B = \{f\in C[0,1] : f$ is differentiable$\}$. Then by Lemmas \ref{lemL} and \ref{lemx2}, we have
\[
A\cap B = \emptyset \quad \text{and} \quad A+B\subseteq A.
\]
From \cite[Theorem 2.6 ]{BFBS}, $A$ is $\mathfrak{c}$-lineable. It follows from Lemma \ref{lemF1} that $A$ is not $(\alpha,\beta)$-spaceable if $\alpha\geq\aleph_0$.
\end{proof}

\section{$(\alpha,\beta)$-spaceability of functions for Hausdorff dimension}
In this section, we investigate the $(\alpha,\beta)$-spaceability of functions in $C[0,1]$ whose graphs possess prescribed Hausdorff dimension. We first establish a restriction theorem (Theorem \ref{thm1}), which is an important tool for solving the main results of this section. It says that for any continuous function $f\in C(K)$ whose graph has large Hausdorff dimension, there exists a large subset $F\subseteq K$ such that the restriction $f|_F$ will be``better behaved" than $f$. Then by applying the restriction theorem, we will build a series of propositions in order to prove the main results of this section.

\begin{definition}
Suppose that $K \subseteq [0,1]$ is a compact set, $f \in C(K)$ and $x \in K$.
If
\[
\dim_H G_f(B(x,r) \cap K) = \dim_H G_f(K)
\]
 for any $r>0$,
then $x$ is said to be a \emph{full-dimensional point of function $f$ for Hausdorff dimension}.
In short, we call $x$ a full-dimensional point of $f$.
\end{definition}
\begin{remark}
By the stability of Hausdorff dimension, full-dimensional points of any function $f\in C(K)$ always exist.
\end{remark}
By an application of the Marstrand slicing theorem, we obtain the following restriction theorem which will be repeatedly used in this section.
 For the convenience of the reader, we recall the Marstrand slicing theorem \cite{Mars}.
\begin{lemma}Given $s\in(1,2]$, let $E$ be a planar set which is $\mathcal{H}^s$ measurable with $0<\mathcal{H}^s(E)<\infty$.
Then at $\mathcal{H}^s$-almost all points $x\in E$ the following is true: for almost all straight lines $l$ through $x$, $\mathcal{H}^{s-1}(E\cap l)<\infty$ and the Hausdorff dimension of $E\cap l$ is equal $s-1$.
\end{lemma}
\begin{theorem}[Restriction Theorem]\label{thm1}
Let $K \subseteq [0,1]$ be a compact set and $s \in (1,2]$.
Then for any $f \in H_s(K)$, there exists a compact subset $F$ of $K$ such that
\begin{enumerate}
    \item $\dim_H F = s-1$,
    \item $f|_F$ is nearly locally Lipschitz,
\end{enumerate}
where $f|_F$ denotes the restriction of $f$ to $F$.
\end{theorem}

\begin{proof}
Let $x_0$ be a full-dimensional point of $f$ (for Hausdorff dimension).
Without loss of generality, we assume that $x_0 \neq 1$ and
\[
\dim_H G_f\left(K \cap B^+(x_0, r)\right) = s
\]
for any $r>0$, where
\[
B^+(x_0, r) = \left\{ x \in \mathbb{R} : 0 \leq x - x_0 < r \right\}.
\]
Set
\[
I_n = \left[ x_0 + \frac{1-x_0}{2^n}, \ x_0 + \frac{1-x_0}{2^{n-1}} \right]
\]
for each $n \in \mathbb{N}$.
Note that $|I_n|=\frac{1-x_0}{2^n}$, where $|I_n|$ denotes the length of the interval $I_n$.
By virtue of the countable stability of the Hausdorff dimension,
there exists a sequence $\{\ell_n\}_{n \geq 1}$ of integers such that
\[
\dim_H G_f\left(K \cap I_{\ell_n}\right) = s_n \to s, \quad n \to \infty
\]
and
\[
\ell_{n+1} - \ell_n \geq 2, \quad s_n > 1 + \frac{s-1}{2}, \quad \forall n \geq 1.
\]
For each $n\in\mathbb{N}$, let $\gamma_n = s_n - \frac{s-1}{2n}>1$. By Besicovitch's theorem on finite measure subsets, there exists a compact planar set $J_n \subseteq G_f(K \cap I_{\ell_n})$ such that
\[
0 < \mathcal{H}^{\gamma_n}(J_n) < +\infty.
\]
Applying the Marstrand slicing theorem to $J_n$, there exists a non-vertical straight
line $L_n$  such that
\[
\dim_H L_n \cap J_n = \gamma_n - 1.
\]
For each $n\in\mathbb{N}$, let $K_n$ be the projection of $L_n \cap J_n$ onto the $x$-axis. Then
\[
G_f(K_n) = L_n \cap J_n.
\]
Thus we obtain a sequence of subsets of $K$, $\{K_n\}_{n\geq 1}$.
Obviously, for any $n\in\mathbb{N}$, $K_n \subseteq K \cap I_{\ell_n}$ is compact and $f$ is Lipschitz on $K_n$;
furthermore,
\[
\dim_H K_n = \dim_H G_f(K_n) = \dim_H L_n \cap J_n = \gamma_n-1
\]
and
\[
2|I_{\ell_n}|\geq\operatorname{dist}(K_n, K_{n+1}) \geq \operatorname{dist}(I_{\ell_n}, I_{\ell_{n+1}}) \geq \frac{|I_{\ell_n}|}{2} = \frac{1-x_0}{2^{\ell_n+1}},
\]
where $\operatorname{dist}$ denotes the Hausdorff metric.
Let
\[
F = \overline{\bigcup_{n \geq 1} K_n}.
\]
In fact, since $\operatorname{dist}(K_n, \{x_0\}) \to 0$ as $n \to \infty$, we have
\[
F = \bigcup_{n \geq 1} K_n \cup \{x_0\}.
\]
Then
\[
\dim_H F = \sup_{n \geq 1} \dim_H K_n = s-1.
\]
Let $x \in F \setminus \{x_0\}$. Without loss of generality, we assume that
$x \in K_{n_0}$ for some $n_0\in\mathbb{N}$, then we have that $f$ is Lipschitz on
\[
F \cap B\left(x, \frac{1-x_0}{2^{\ell_{n_0}+1}}\right) \subseteq K_{n_0}.
\]
Hence $f$ is nearly locally Lipschitz on $F$.
\end{proof}

\begin{remark}\label{remark1}
1) From the proof of Theorem \ref{thm1}, it is easy to see that if $x_0$ is a
full-dimensional point of $f\in H_s(K)$, then for any $r>0$, there exists a
compact set $K_r \subseteq B(x_0,r) \cap K$ such that $f$ is nearly locally Lipschitz on $K_r$ and
\[
\dim_H K_r = s-1.
\]

\medskip
2) Theorem \ref{thm1} reveals quite a counterintuitive phenomenon:
the larger the Hausdorff dimension of the graph of $f$, that is,
the less regular $f$ is (in the sense of dimension), the larger the
set on which $f$ can be regular in the sense of nearly locally Lipschitz
continuity. In particular, if $\dim_H G_f(K) > s$ (or $\mathcal{H}^{s}(G_f(K)) > 0$), then there
exists a compact set $K_1 \subseteq K$ such that
\[
0 < \mathcal{H}^{s}(G_f(K_1)) < \infty.
\]
In fact, applying Marstrand slicing theorem to $G_f(K_1)$,
we get that there exists a non-vertical straight line $L$ such that
\[
\dim_H G_f(K_1) \cap L = s-1.
\]
Let
\[
F = \pi(G_f(K_1) \cap L) \subseteq K,
\]
where $\pi$ stands for the projection mapping onto the $x$-axis for planar sets.
Thus we get that
\[
\dim_H F = s-1, \quad \text{and} \quad f \text{ is Lipschitz on } F.
\]

\medskip
3) It remains unclear whether the present theorem is sharp, that is,
whether there exist a function $f \in H_{<s}(K)$ such that for any
$F \subseteq K$ with $\dim_H F = s-1$, $f$ is not nearly locally Lipschitz on $F$.
\end{remark}

\begin{definition}
Let $K \subseteq [0,1]$ be a compact set and let $f \in C(K)$, $x_0 \in K$, $\beta \in (0,+\infty)$. Set
\[
V_\beta = \left\{ (x,y) \in \mathbb{R}^2 :x\neq x_0,\ \left| \frac{y - f(x_0)}{x - x_0} \right| \leq \beta \right\}.
\]
If
\[
\dim_H \left( V_\beta \cap G_f(K \cap B(x_0,r)) \right) = \dim_H G_f(K),
\]
 for any $r>0$, we call $x_0$ a $\beta$-regular full-dimensional point of $f$.
If a function $f$ admits a $\beta$-regular full-dimensional point, we say that $f$ is a $\beta$-regular full-dimensional function. More generally, a function $f$ is said to be regular full-dimensional if there exists some $\beta \in (0,+\infty)$ such that $f$ is $\beta$-regular full-dimensional.
\end{definition}

The following theorem states that if $f$ is regular full-dimensional, then $f$ is Lipschitz on a larger subset.
\begin{theorem}
Let $s \in (1,2]$ and let $K \subseteq [0,1]$ be a compact set. If $f \in H_s(K)$ is a regular full-dimensional function, then there exists a compact subset $F$ of $K$ such that
\begin{enumerate}
    \item $\dim_H F = s-1$;
    \item $f$ is Lipschitz on $F$.
\end{enumerate}
\end{theorem}

\begin{proof}
By the hypotheses of the theorem, there exists $\beta \in (0,+\infty)$ and $x_0 \in K$ such that $x_0$ is a $\beta$-regular full-dimensional point of $f$. We assume that $x_0 \neq 1$ (otherwise, we consider the opposite triangular region in the following).
By virtue of the stability of Hausdorff dimension, we may assume that for any $r>0$,
\begin{equation}\label{tm2eq1}
\dim_H \left( V_\beta^+ \cap G_f(K \cap B(x_0,r)) \right) = \dim_H G_f(K) = s,
\end{equation}
where
\[
V_\beta^+ = \left\{ (x,y) \in \mathbb{R}^2 : x > x_0, \ \left| \frac{y - f(x_0)}{x - x_0} \right| \leq \beta \right\}.
\]

For each $n \in \mathbb{N}$, let
\[
I_n = \left[ x_0 + \frac{1-x_0}{2^n},\ x_0 + \frac{1-x_0}{2^{n-1}} \right]
\quad \text{and} \quad
G_n = G_f(I_n\cap K) \cap V_\beta^+.
\]
Then by Equation (\ref{tm2eq1}), there exists a strictly increasing sequence of integers $\{\ell_k\}_{k\geq 1}$ such that
\[
\dim_H G_{\ell_k} = s_k \to s \quad \text{as } k \to \infty,
\]
and
\[
s_k > 1, \quad  \ell_{k+1} - \ell_k \geq 2
\]
for any $k\in\mathbb{N}$.
Let $k_0 \in \mathbb{N}$ be sufficiently large so that
\[
\gamma_k = s_k - \frac{1}{k} > 1
\]
for all $k \geq k_0$. Note that $\gamma_k \to s$ as $k \to \infty$. By Besicovitch's theorem, for each $k \geq k_0$, there exists a compact subset $J_k \subseteq G_{\ell_k}$ such that
\[
0 < \mathcal{H}^{\gamma_k}(J_k) < +\infty.
\]
Further applying the Marstrand slicing theorem, there exists a line $L_k$ with slope $t_k \in (0,1)$ such that
\[
\dim_H (L_k \cap J_k) = \gamma_k - 1.
\]
Let $F_k = \pi(L_k \cap J_k)$, for each $k \geq k_0$, i.e., $F_k$ is the projection of $L_k \cap J_k$ onto the $x$-axis. Then
\[
G_f(F_k) = L_k \cap J_k.
\]
Thus,
\[
\dim_H F_k = \dim_H (L_k \cap J_k) = \gamma_k - 1.
\]
Note that $F_k \subseteq I_{\ell_k}\cap K\subseteq I_{\ell_k}$. We define the compact set $F$ as
\[
F = \overline{\bigcup_{k \geq k_0} F_k}.
\]
In fact,
\[
F = \bigcup_{k \geq k_0} F_k \cup \{x_0\},
\]
and therefore
\[
\dim_H F = \sup_{k \geq k_0} \dim_H F_k = s - 1.
\]

Next, we prove that $f$ is Lipschitz on $F$. It suffices to prove that $f$ is Lipschitz on $\bigcup_{k \geq k_0} F_k$.
Let $x,y \in \bigcup_{k \geq k_0} F_k$ be arbitrary.
If $x,y \in F_{n_0}$ for some $n_0 \geq k_0$, by the definition of $F_{n_0}$, we have
\[
|f(x) - f(y)| \leq t_{n_0} |x-y| < |x-y|.
\]
If
\[
x \in F_j \subseteq I_{\ell_j}, \quad y \in F_i \subseteq I_{\ell_i}
\]
 for some $j > i \geq k_0$,
then
\[
|x-y| \geq \frac{1}{2} |I_{\ell_i}| \quad (\text{note that } \ell_j - \ell_i \geq 2).
\]
Furthermore, we have
\[
\begin{aligned}
|f(x) - f(y)| &\leq |f(x) - f(x_0)| + |f(y) - f(x_0)| \\
&\leq \beta \left(|x-x_0| + |y-x_0|\right) \\
&\leq 3\beta |I_{\ell_i}| < 6\beta |x-y|.
\end{aligned}
\]
This proves that $f$ is Lipschitz on $\bigcup_{k \geq k_0} F_k$.
Since $F = \overline{\bigcup_{k \geq k_0} F_k}$, $f$ is Lipschitz on $F$.
\end{proof}

\begin{corollary}\label{cor1}
Let $s \in (1,2]$ and let $K \subseteq[0,1]$ be a compact set.
If $f \in H_s(K)$ has infinitely many full-dimensional points, then there exists a sequence of compact subsets $\{K_j\}_{j \geq 1}$ satisfying the following properties:
\begin{enumerate}
    \item $\dim_H K_j = s-1, \ \forall j \geq 1$;
    \item $\mathrm{conv}(K_i) \cap \mathrm{conv}(K_j) = \emptyset, \ \forall i \neq j$;
    \item $f$ is nearly locally Lipschitz on each $K_j$.
\end{enumerate}
Furthermore,
\[
f\Big|_{K \setminus \bigcup_{j \geq 1} K_j} \in H_s\Big(K \setminus \bigcup_{j \geq 1} K_j\Big),
\]
that is, $\dim_HG_f(K\setminus\bigcup_{j\geq 1}K_j)=s$,
where $\mathrm{conv}(A)$ denotes the convex hull of the set $A$.
\end{corollary}

\medskip
\begin{proof} Suppose that $\{x_i\}_{i \geq 1}$ is a sequence of pairwise distinct full-dimensional points of $f$. Take $r_i>0$ such that
\[
B(x_i, r_i) \cap B(x_j, r_j) = \emptyset.
\]
By the Remark \ref{remark1}, it suffices to take
\[
K_i = F \cap \overline{B}\Big(x_i, \frac{r_i}{2}\Big),
\]
where the set  $F$ is constructed in Theorem \ref{thm1}.
\end{proof}

\medskip
The following proposition is of crucial importance for the proof of our main result, Theorem \ref{thmn1} in this section.
\begin{proposition}\label{prop1}
Let $s \in (1,2]$ and let $K \subseteq [0,1]$ be a compact set with $\mathcal{H}^{s-1}(K) > 0$.
Let $f, g \in H_s(K)$ be linearly independent and satisfy $\operatorname{span}\{f,g\} \subseteq H_s(K) \cup \{0\}$.
Then there exists a sequence of compact subsets $\{K_j\}_{j\ge1}$ with the following properties:
\begin{enumerate}
    \item $\dim_H K_j = s-1$, $\forall j\ge1$;
    \item $\operatorname{conv}(K_i) \cap \operatorname{conv}(K_j) = \emptyset$, $\forall i\ne j$;
    \item for any $(a,b) \in \mathbb{R}^2 \setminus \{(0,0)\}$, we have
    \[
    (af+bg) \Big|_{K \setminus \bigcup_{j\ge1} K_j} \in H_s\Big(K \setminus \bigcup_{j\ge1} K_j\Big).
    \]
\end{enumerate}
\end{proposition}

\begin{proof}
We will prove it by considering the following two cases.

\medskip
\noindent \textbf{Case 1.} At least one of $f$ and $g$ has infinitely many full-dimensional points.

Without loss of generality, assume that $f$ has infinitely many full-dimensional points.
From Corollary \ref{cor1}, we may choose a sequence of compact subsets $\{K_j^{(1)}\}_{j\ge1}$ satisfying the properties $(1)$ and $(2)$ of this proposition. Additionally, $f$ is nearly locally Lipschitz on each $K_j^{(1)}$. We also assume that $\mathcal{H}^{s-1}\big(K \setminus \bigcup_{j\ge1} K_j^{(1)}\big) > 0$; if this fails,  we may instead work with an appropriate compact subset of $K_j^{(1)}$ (for some $j$) with Hausdorff dimension $s-1$ and positive $\mathcal{H}^{s-1}$-measure

\medskip
\noindent\textbf{a1)} Suppose $\dim_H G_g\big(\bigcup_{j\ge1} K_j^{(1)}\big) < s$.

Obviously, by Lemma \ref{lemL}, we have
\[
\begin{aligned}
\dim_H G_{af+bg}\Big(\bigcup_{j\ge1} K_j^{(1)}\Big) &=\sup_{j\geq1}\dim_H G_{af+bg} (K^{(1)}_j) \\
&=\sup_{j\geq1}\dim_H G_{bg} (K^{(1)}_j)\leq\sup_{j\geq1}\dim_H G_{g} (K^{(1)}_j) \\
&=\dim_H G_g\Big(\bigcup_{j\ge1} K_j^{(1)}\Big)< s
\end{aligned}
\]
for any $(a,b) \in \mathbb{R}^2$.
Hence we have
\[
\dim_H G_{af+bg}\Big(K \setminus \bigcup_{j\ge1} K_j^{(1)}\Big) = s, \quad \forall (a,b) \in \mathbb{R}^2 \setminus \{(0,0)\}.
\]
That is, $\{K_j^{(1)}\}_{j\ge1}$ satisfies all the properties of the proposition. At this stage, just set $K_j = K_j^{(1)}$ for each $j\geq1$.

\medskip
\noindent \textbf{a2)} Suppose $\dim_H G_g\big(\bigcup_{j\ge1} K_j^{(1)}\big) = s$.

Then we have
\begin{equation}\label{propeq1}
\dim_H G_{af+bg}\Big(\bigcup_{j\ge1} K_j^{(1)}\Big)
= \dim_H G_g\Big(\bigcup_{j\ge1} K_j^{(1)}\Big) = s
\end{equation}
for any $(a,b) \in \mathbb{R}^2$ with $b \ne 0$.
Since $\mathcal{H}^{s}\big(K \setminus \bigcup_{j\ge1} K_j^{(1)}\big) > 0$, there exists a sequence $\{K_i^{(2)}\}_{i\ge1}$ of compact sets satisfying the following properties:
\[
K_i^{(2)} \subseteq K \setminus \bigcup_{j\ge1} K_j^{(1)},\ \ \dim_H K_i^{(2)} = s-1
\]
for all $i\ge1$, and
\[
\operatorname{conv}(K_i^{(2)}) \cap \operatorname{conv}(K_j^{(2)}) = \emptyset, \ \forall i\ne j,
\]
\[
\mathcal{H}^{s-1}\Big(K \setminus \Big(\bigcup_{j\ge1} K_j^{(1)} \cup \bigcup_{i\ge1} K_i^{(2)}\Big)\Big) > 0.
\]

If $\dim_H G_f\big(\bigcup_{i\ge1} K_i^{(2)}\big) < s$, then from the fact that $f$ is nearly locally Lipschitz on each $K_j^{(1)}$, we get
\[
\dim_H G_f\Big(\bigcup_{j\ge1} K_j^{(1)}\Big)
= \sup_{j\ge1} \dim_H G_f(K_j^{(1)}) = s-1 < s.
\]
Therefore,
\[
\dim_H G_f\Big(K \setminus \Big(\bigcup_{j\ge1} K_j^{(1)} \cup \bigcup_{i\ge1} K_i^{(2)}\Big)\Big) = s.
\]
Combining these with Equation (\ref{propeq1}), we get that
\[
\dim_H G_{af+bg}\Big(K \setminus \bigcup_{j\ge1} K_j^{(2)}\Big) = s
\]
for any $(a,b) \in \mathbb{R}^2 \setminus \{(0,0)\}$.
At this point, set $K_j = K_j^{(2)}$ for each $j$. Then $\{K_j\}_{j\geq 1}$ satisfies all the properties of this proposition.

 If $\dim_H G_f\big(\bigcup_{i\ge1} K_i^{(2)}\big) = s$, then take any sequence $\{K_j\}_{j\ge1}$ of compact sets such that
 \[K_j \subseteq K \setminus \Big(\bigcup_{j\ge1} K_j^{(1)} \cup \bigcup_{i\ge1} K_i^{(2)}\Big),\]
\[\dim_H K_j = s-1 \]
for all $j\geq1$ and
\[\operatorname{conv}(K_i) \cap \operatorname{conv}(K_j)=\emptyset, \ \forall i\neq j.\]
It is easy to see that $\{K_j\}_{j\geq 1}$ also satisfies property $(3)$ of the proposition.

\medskip
\noindent \textbf{Case 2.} Both $f$ and $g$ have finitely many full-dimensional points.

\medskip
Define $P_{\{f,g\}}$ as follows:
\[
P_{\{f,g\}} = \left\{ x \in K : \exists (a,b)\neq(0,0)\ \text{s.t. } x \text{ is a full-dimensional point of } af+bg \right\},
\]
that is, $P_{\{f,g\}}$ is the set of full-dimensional points of functions belonging to $\operatorname{span}\{f,g\} \setminus \{0\}$.

\medskip
\noindent\textbf{b1)} Suppose $\# P_{\{f,g\}} < \infty$.

In this case, without loss of generality, write $P_{\{f,g\}}=\{0<x_1 < x_2 < \cdots < x_n<1\}$.
Let $I_k$ be the open interval $(x_{k-1},x_k)$, $k=1,2,\dots,n+1$, where $x_0=0$, $x_{n+1}=1$.
Since $\mathcal{H}^{s-1}(K) > 0$, there exists $ k_0$ such that
\[
\mathcal{H}^{s-1}(K \cap I_{k_0}) > 0.
\]
Take a compact subset $F$ of $K \cap I_{k_0}$ with $0 < \mathcal{H}^{s-1}(F) < \infty$. Then
\[
 \dim_H G_{af+bg}(F) < s
\]
for any $(a,b) \in \mathbb{R}^2 \setminus \{(0,0)\}$.
(Otherwise, $P_{\{f,g\}} \cap F \ne \emptyset$, which leads to a contradiction.)
Hence we can take a sequence $\{K_j\}_{j\geq 1}$ of compact subsets of $F$ satisfying the following conditions:
\[
\dim_H K_j = s-1, \ \forall j\ge1
\]
and
\[
\operatorname{conv}(K_i) \cap \operatorname{conv}(K_j) = \emptyset, \ \forall i\ne j.
\]
Then $\{K_j\}_{j\ge1}$ is exactly the required sequence.

\medskip
\noindent\textbf{ b2)} Suppose $\# P_{\{f,g\}} = +\infty$.

Take a sequence of mutually distinct points $\{x_n\}_{n\ge1} \subseteq P_{\{f,g\}}$. For each $n\ge1$,
choose $(a_n,b_n) \in \mathbb{R}^2 \setminus \{(0,0)\}$ such that $x_n$ is a full-dimensional point of $a_n f + b_n g$.
Without loss of generality, by taking a subsequence if necessary, we may assume that $\{(a_n,b_n)\}_{n\ge1}$ is pairwise linearly independent. If this were not possible, there would exist some $(a',b') \in \mathbb{R}^2 \setminus \{(0,0)\}$ such that $a'f+b'g$ has infinitely many full-dimensional points, then considering the pairs $\{a'f+b'g, g\}$ or $\{a'f+b'g, f\}$ reduces us back to Case 1.
We may also assume that $a_n \ne 0$, $\forall n\ge1$ (or, symmetrically, $b_n \ne 0$, $\forall n\ge1$) by passing to a further subsequence.

For each $n \in \mathbb{N}$, take $r_n>0$ such that the family $\{B(x_n,r_n)\}_{n\ge1}$ is pairwise disjoint.
By Remark \ref{remark1}, there exists a compact set $F_n \subseteq B(x_n,r_n)\cap K$ satisfying
\begin{equation}\label{propeq2}
\dim_H F_n = s-1, \quad \mathcal{H}^{s-1}\Big(K \setminus \bigcup_{n\ge1} F_n\Big) > 0,
\end{equation}
and such that $a_n f + b_n g$ is nearly locally Lipschitz on $F_n$.
Since $a_n \ne 0$ for any $n\in\mathbb{N}$, we have the decomposition
\[
af + bg = \frac{a}{a_n}(a_n f + b_n g) + \Big(b - \frac{b_n}{a_n}a\Big)g.
\]
Consequently,
\begin{equation}\label{propeq3}
\dim_H G_{af+bg}(F_n) =
\begin{cases}
\dim_H G_g(F_n), & a b_n-b a_n  \ne 0, \\
s-1, & a b_n-b a_n = 0
\end{cases}
\end{equation}
for any $(a,b)\in\mathbb{R}^2$.
If $\dim_H G_g\big(\bigcup_{n\ge1} F_n\big) < s$, then for any given $(a,b) \in \mathbb{R}^2$, we have
\[
\dim_H G_{af+bg}\Big(\bigcup_{n\ge1} F_n\Big)
= \sup_{n\ge1} G_{af+bg}(F_n) < s.
\]
This further yields
\[
\dim_H G_{af+bg}\Big(K \setminus \bigcup_{n\ge1} F_n\Big) = s
\]
for any $(a,b)\in \mathbb{R}^2 \setminus \{(0,0)\}$.
Combining this with Equation (\ref{propeq2}) and disjointness of the balls $B(x_n,r_n)$, it suffices to set $K_j=F_j$, for all $j\in\mathbb{N}$.

If $\dim_H G_g\big(\bigcup_{n\ge1}F_n\big)=s$,
we may assume without loss of generality that $\dim_H G_g(F_n) < s$ for any $n\in\mathbb{N}$. Otherwise, suppose $\dim_H G_g(F_n)=s$ for some integer $n$. By Theorem \ref{thm1}, we consider the restriction $g|_{F_n}$, which yields a compact subset $F_n'\subseteq F_n$ such that $g$ is nearly locally Lipschitz on $F'_n$  and $\dim_H F'_n=s-1$. Consequently, $\dim_H G_g(F'_n) < s$. If $\dim_H G_g(F_n)<s$ already holds, we simply set $F_n'=F_n$ and revisit the value of $\dim_H G_g\big(\bigcup_{n\ge1}F_n'\big)$). Since $\{(a_n,b_n)\}_{n\ge1}$ is pairwise linearly independent, then for any given $(a,b)\in \mathbb{R}^2\setminus\{(0,0)\}$, the inequality
\[
 ab_n -ba_n \ne 0
\]
hols for all but at most one index $n$ (denoted by $n_0$ if such an index exists).
Therefore, in view of Equation (\ref{propeq3}),
\[
\begin{aligned}
\dim_H G_{af+bg}\Big(\bigcup_{n\ge1} F_n\Big)
&= \sup_n \dim_H G_{af+bg}(F_n) \\
&= \max\Big\{s-1, \sup_{n\ne n_0} \dim_H G_g(F_n)\Big\} \quad (\text{if } n_0 \text{ exists}) \\
&= \dim_H G_g\big(\bigcup_{n\ge1} F_n\big) = s
\end{aligned}
\]
for any $(a,b)\in \mathbb{R}^2\setminus\{(0,0)\}$.
Since $\mathcal{H}^{s-1}\big(K \setminus \bigcup_{n\ge1} F_n\big) > 0$, we may select a sequence of compact sets $\{K_j\}_{j\ge1}$ with the following properties:
\[
K_j \subseteq K \setminus \bigcup_{n\ge1} F_n,
\]
\[
\dim_H K_j = s-1,
\]
for any $ j\ge1$ and
\[
\operatorname{conv}(K_i) \cap \operatorname{conv}(K_j) = \emptyset, \ \forall i\ne j.
\]
It follows immediately that
\[
\dim_H G_{af+bg}\Big(K \setminus \bigcup_{j\ge1} K_j\Big)
\ge \dim_H G_{af+bg}\Big(\bigcup_{j\ge1} F_j\Big) = s
\]
holds for any $(a,b) \in \mathbb{R}^2 \setminus \{(0,0)\}$.
\end{proof}

\medskip
The following proposition plays a crucial role in the proof of our main result, Theorem \ref{thm2} presented in this section.
\begin{proposition}\label{prop2}
Let $s\in(1,2]$, $n\in\mathbb{N}$, and let $K\subseteq[0,1]$ be a compact set with $\mathcal{H}^{s-1}(K) > 0$. Suppose that $\{f_i\}_{i=1}^n\subseteq H_s(K)$ is linearly independent and
$\operatorname{span}\{f_i\}^n_{i=1}\subseteq H_s(K)\cup\{0\}$. Then for any $m\in\mathbb{N}$, there exist compact subsets $\{K_j\}_{j=1}^m$ satisfying the following properties:
\begin{enumerate}
\item $\dim_H K_j=s-1$, $\forall$ $j=1,2,\dots,m$;
\item $\operatorname{conv}(K_i)\cap \operatorname{conv}(K_j)=\emptyset$, whenever $i\neq j$;
\item for any $(a_1,\dots,a_n)\in\mathbb{R}^n\setminus\{(0,\dots,0)\}$, we have
\[
\left.\sum_{i=1}^n a_i f_i\right|_{K\setminus\bigcup_{j=1}^m K_j}
\in H_s\Big(K\setminus\bigcup_{j=1}^m K_j\Big).
\]
\end{enumerate}
\end{proposition}

\begin{proof}
According to the number of full-dimensional points of the functions in $\{f_i\}^n_{i=1}$, we split the proof into the following two cases.

\medskip
\noindent\textbf{Case 1.}
At least one function in $\{f_i\}_{i=1}^n$ admits infinitely many full-dimensional points. Without loss of generality, we take such a function to be $f_1$.

By Corollary \ref{cor1}, there exist $m$ compact subsets $\{K_j^{(1)}\}_{j=1}^m$ such that
\begin{itemize}
\item $\dim_H K_j^{(1)}=s-1,\ \forall\,j=1,2,\dots,m$;
\item $\operatorname{conv}(K_i^{(1)})\cap \operatorname{conv}(K_j^{(1)})=\emptyset$ whenever $i\neq j$;
\item $f_1$ is nearly locally Lipschitz on each $K_j^{(1)}$.
\end{itemize}

\begin{claim}\label{claim}
For each $j\in\{1,2,\dots,m\}$, there exists a compact set $K_j\subseteq K_j^{(1)}$ such that
\[
\dim_H K_j=s-1 \quad \text{and} \quad
\left.\sum_{i=1}^n a_i f_i\right|_{K_j}\in H_{<s}(K_j)
\]
for any $(a_1,\dots,a_n)\in\mathbb{R}^n$.
\end{claim}

\noindent\textbf{Proof of the Claim.}
Fix an arbitrary $j\in\{1,2,\dots,m\}$.

\textbf{a1)} Assume that $\displaystyle\left.\sum_{i=2}^n a_i f_i\right|_{K_j^{(1)}}
\in H_{<s}(K_j^{(1)})$ holds for any $(a_2,\dots,a_n)\in\mathbb{R}^{n-1}$.

Since $f_1$ is nearly locally Lipschitz on $K_j^{(1)}$, by applying Lemma \ref{lemL}, we have
\[
\left.\sum_{i=1}^n a_i f_i\right|_{K_j^{(1)}}\in H_{<s}(K_j^{(1)})
\]
for any $(a_1,\dots,a_n)\in\mathbb{R}^n$.
In this case, we simply set $K_j=K_j^{(1)}$.

\textbf{a2)} Suppose that there exists $(a_2^{(1)},\dots,a_n^{(1)})\in\mathbb{R}^{n-1}$ such that
\[
\left.\sum_{i=2}^n a_i^{(1)} f_i\right|_{K_j^{(1)}}\in H_s(K_j^{(1)}).
\]
In this case, we have $(a_2^{(1)},\dots,a_n^{(1)})\neq(0,\dots,0)$. Without loss of generality, we assume that $a_2^{(1)}\neq0$. By Theorem \ref{thm1}, there exists a compact subset $K_j^{(2)}\subseteq K_j^{(1)}$ such that
\[
\dim_H K_j^{(2)}=s-1
\]
and
\[
\sum_{i=2}^n a_i^{(1)} f_i \text{ is nearly locally Lipschitz on } K_j^{(2)}.
\]
Observe that
\[
\sum_{i=2}^n a_i f_i
= \frac{a_2}{a_2^{(1)}} \sum_{i=2}^n a_i^{(1)} f_i
+ \sum_{i=3}^n \left(a_i - \frac{a_2}{a_2^{(1)}}a_i^{(1)}\right) f_i
\]
for any $(a_2,\dots,a_n)\in\mathbb{R}^{n-1}$.
If $\displaystyle\left.\sum_{i=3}^n a_i f_i\right|_{K_j^{(2)}}
\in H_{<s}(K_j^{(2)})$ holds for all $(a_3,\dots,a_n)\in\mathbb{R}^{n-2}$, it suffices to take $K_j=K_j^{(2)}$.

Otherwise, we repeat the procedures in a2) and a1). Then there necessarily exists some integer $k_0\in\{1,2,\dots,n\}$, compact sets $\{K_j^{(k)}\}_{k=1}^{k_0}$ and coefficient vectors
\[
\big(a_k^{(k-1)},a_{k+1}^{(k-1)},\dots,a_n^{(k-1)}\big)
\in\mathbb{R}^{n-k+1}\setminus\{(0,\dots,0)\},
\]
where we always assume that $a_k^{(k-1)}\neq0$ without loss of generality, such that
\begin{itemize}
\item $K_j^{(k)}\subseteq K_j^{(k-1)}$ ;
\item $\dim_H K_j^{(k)}=s-1$;
\item $\displaystyle\left.\sum_{i=k}^n a_i^{(k-1)} f_i\right|_{K_j^{(k-1)}}
\in H_s\big(K_j^{(k-1)}\big)$;
\item $\displaystyle\sum_{i=k}^n a_i^{(k-1)} f_i$ is nearly locally Lipschitz on $K_j^{(k)}$
\end{itemize}
for any $k=2,\dots,k_0$.
Furthermore, if $k_0<n$, we additionally have
\[
\left.\sum_{i=k_0+1}^n a_i f_i\right|_{K_j^{(k_0)}}
\in H_{<s}\big(K_j^{(k_0)}\big),\quad \forall (a_{k_0+1},\dots,a_n)\in\mathbb{R}^{n-k_0}.
\]
It is straightforward to verify that $K_j^{(k_0)}$ is exactly the compact subset we need. Indeed, for any $(a_1,\dots,a_n)\in\mathbb{R}^n$, there exists a vector $(\beta_2,\dots,\beta_{k_0})$ such that
\[
\sum_{i=1}^n a_i f_i
= a_1 f_1 + \sum_{k=2}^{k_0}\beta_k \sum_{i=k}^n a_i^{(k-1)} f_i
+ \sum_{i=k_0+1}^n \beta_i f_i.
\]
Thus by Lemma \ref{lemL}, we have
\[
\left.\sum_{i=1}^n a_i f_i\right|_{K_j^{(k_0)}}\in H_{<s}(K_j^{(k_0)}).
\]
Set $K_j=K_j^{(k_0)}$ for each $j$. The claim is thus proved.

\medskip
We now return to the proof of Case 1.
By the construction of $\{K_j\}_{j=1}^m$, the properties (1) and (2) in this proposition are obviously satisfied. Combining this with Claim \ref{claim} and the stability of Hausdorff dimension, we conclude that property (3) holds as well.

\medskip
\noindent\textbf{Case 2.}
Every function in $\{f_i\}_{i=1}^n$ admits finitely many full-dimensional points.

Define
\[
P_{\{f_i\}} = \Big\{ x\in K : \exists\,(a_i)^n_{i=1}\neq(0,\dots,0)
\ \text{s.t.}\ x \text{ is a full-dimensional point of } \sum_{i=1}^n a_i f_i \Big\}.
\]

\textbf{b1)} Suppose $\#P_{\{f_i\}}<+\infty$.

By an argument analogous to that in Proposition \ref{prop1}, we can construct $m$ compact sets $\{K_j\}_{j=1}^m$ which satisfy all the assertions of the proposition.

\medskip
\textbf{b2)} Suppose $\#P_{\{f_i\}}=+\infty$.

Take $m$ pairwise distinct points $x_1,x_2,\dots,x_m$ from $P_{\{f_i\}}$. For each $j=1,\dots,m$, choose a coefficient vector
\[
(a_1^{(j)},a_2^{(j)},\dots,a_n^{(j)})
\in\mathbb{R}^n\setminus\{(0,\dots,0)\}
\]
such that $x_j$ is a full-dimensional point of $\displaystyle\sum_{i=1}^n a_i^{(j)} f_i$.

Without loss of generality, we assume that $a_1^{(j)}\neq0$ for all $j\in\{1,2,\dots,m\}$.
Chooe $r_j>0$ such that $\{B(x_j,r_j)\}^m_{j=1}$ is pairwise disjoint.
By Remark \ref{remark1}, for each $j\in\{1,2,\dots,m\}$, there exists a compact set $K_j'$ such that
\begin{equation}\label{prop2eq1}
K_j'\subseteq B(x_j,r_j)\cap K,
\quad\  \text{and thus}\ \operatorname{conv}(K_i')\cap \operatorname{conv}(K_j')=\emptyset,\ \forall\,i\neq j,
\end{equation}
\[
\dim_H K_j'=s-1
\]
and the function
\[
g_j=\sum_{k=1}^n a_k^{(j)} f_k\]
is nearly locally Lipschitz on $K_j'$.
For each $j\in\{1,2,\dots,m\}$, consider the family $\{g_j,f_2,\dots,f_n\}$
(note that $\operatorname{span}\{g_j,f_2,\dots,f_n\}=\operatorname{span}\{f_1,f_2,\dots,f_n\}$), by  applying Claim \ref{claim},  we obtain a compact subset $K_j\subseteq K_j'$ such that
\begin{equation}\label{prop2eq2}
\dim_H K_j=s-1
\end{equation}
and
\[
 \left.\Big(a_1 g_j + \sum_{k=2}^n a_k f_k\Big)\right|_{K_j}\in H_{<s}(K_j),\ \forall\ (a_1,\dots,a_n)\in\mathbb{R}^n,
\]
which further implies
\[
\left.\sum_{k=1}^n a_k f_k\right|_{K_j}\in H_{<s}(K_j)
\]
for all $(a_1,\dots,a_n)\in\mathbb{R}^n$.
Furthermore, we have
\begin{equation}\label{prop2eq3}
\left.\sum_{k=1}^n a_k f_k\right|_{\bigcup_{j=1}^m K_j}
\in H_{<s}\Big(\bigcup_{j=1}^m K_j\Big),\ \forall\ (a_1,\dots,a_n)\in\mathbb{R}^n.
\end{equation}
It follows from Equations (\ref{prop2eq1}), (\ref{prop2eq2}), (\ref{prop2eq3}) and the hypotheses of the proposition that the family $\{K_j\}_{j=1}^m$ fulfills all the required assertions of the proposition. This completes the proof.

\end{proof}

\begin{remark} From the constructions of $\{K_j\}^m_{j=1}$ in the above proposition, it is straightforward to see that
\[
\left.\sum_{k=1}^n a_k f_k\right|_{\bigcup_{j=1}^m K_j}
\in H_{<s}\Big(\bigcup_{j=1}^m K_j\Big)
\]
for any $(a_1,\dots,a_n)\in\mathbb{R}^n$.
\end{remark}
\medskip
Next, we recall the definition of a basic sequence and state an associated result.
\begin{definition} A sequence $(e_n)_{n\geq 1}$ in a Banach space $X$ is called a (Schauder) basic sequence if it constitutes a Schauder basis for $\overline{\operatorname{span}}(e_n)_{n\geq 1}$.
\end{definition}
 The following lemma is probably folklore; nevertheless, we provide its proof for the sake of completeness.
\begin{lemma}\label{lembasic}
Let $X$ be a Banach space, and let $(e_n)_{n \geq 1}$ be a basic sequence. Suppose that $\{x_i\}^m_{i=1} \subseteq X$ and  $\{\overline{x}_i = x_i + X_0: i=1,\dots,m\}$ is linearly independent in the quotient space $X / X_0$. Then the sequence $S = \{x_1,\dots, x_m, e_1, e_2, \dots\}$ is a basic sequence in $X$.
\end{lemma}
\begin{proof}
By the linear independence of $\{\overline{x}_i\}^m_{i=1}$ in $X / X_0$, we have
\[\operatorname{span}\{x_i\}^m_{i=1} \cap X_0 = \{0\}.\]
Thus the sum
\[
\operatorname{span}\{x_i\}^m_{i=1} + X_0
\]
is a direct sum. Let $Y' = \operatorname{span}\{x_i\}^m_{i=1}\oplus X_0$. Clearly, $Y'$ is dense in $Y = \overline{\operatorname{span}} S$. Since $\operatorname{span}\{x_i\}^m_{i=1}$ is finite dimensional and $X_0$ is closed, then $Y'$ is closed, and hence
\[
Y = Y' = \operatorname{span}\{x_i\}^m_{i=1} \oplus X_0.
\]
By a straightforward application of the closed graph theorem, the linear projection $P$ from $Y$ onto $\operatorname{span}\{x_i\}^m_{i=1}$ with kernel $X_0$ is continuous.

Let $C_P = \|P\|$, then $\|I - P\| \leq 1 + C_P$. Since $(e_n)_{n\geq 1}$ is a basic sequence,
by the basic sequence test formula, there exists a constant $C_0 > 0$ such
that
\begin{equation}\label{lembeq1}
\left\| \sum_{n=1}^k a_n e_n \right\| \leq C_0 \left\| \sum_{n=1}^N a_n e_n \right\|
\end{equation}
for any $N \in \mathbb{N}$, all scalars $(a_i)_{i=1}^N$ and all $1 \leq k < N$.
By the equivalence of norms on finite-dimensional spaces, there exists
a constant $C_1 > 0$ such that
\begin{equation}\label{lembeq2}
 C^{-1}_1 \left\|\sum^m_{i=1}b_i x_i\right\|\leq\max_{1\leq i\leq m}\{|b_i|\} \leq C_1 \left\|\sum^m_{i=1}b_i x_i\right\|
\end{equation}
for any $(b_1,\dots,b_m) \in \mathbb{K}^m$.

Finally, we verify the basic sequence test formula for $S$.
We write the elements of $S$ as $y_i = x_i$, $1\leq i\leq m$ and $y_{m+i} = e_i$, $i\geq 1$, and let $(a_i)_{i \geq 1}$
be a sequence of scalars, and let
\[
z_n = \sum_{i=1}^n a_i y_i
\]
for each $n\geq 1$. We need to prove that there exists a constant $C$ (independent of scalar sequence
$(a_i)_{i\geq 1}$ and the integer $N$) such that
\begin{equation}\label{lembeq3}
\|z_n\| \leq C \|z_N\|
\end{equation}
holds for all integers $ 1 \leq n < N$.

 If $N\leq m$, then by inequality (\ref{lembeq2}), we deduce
\begin{equation}
\begin{aligned}\label{lembeq4}
\|z_n\|=\left\|\sum^n_{i=1}a_iy_i\right\|&\leq C_1\max_{1\leq i\leq n}\{|a_i|\}\leq C_1\max_{1\leq i\leq N}\{|a_i|\}\\
&\leq C^2_1\left\|\sum^N_{i=1}a_iy_i\right\|=C^2_1\|z_N\|.
\end{aligned}
\end{equation}
If $N> m$ and $n\leq m$,
write
\[
z_n = \sum^n_{i=1}a_ix_i, \quad z_N = z_m + y,
\]
where \( y = \sum_{i=1}^{N-m} a_{m+i} e_{i} \).
Clearly, $ P(z_N) = z_m $ and $ (I - P)(z_N) = y $. Applying estimate (\ref{lembeq4}), we obtain
\begin{equation}\label{lembeq5}
\|z_n\|\leq C^2_1\|z_m\|= C^2_1\|P(z_N)\|\leq C^2_1C_p\|z_N\|.
\end{equation}
If $N> m$ and $n> m$,  decompose
\[
z_n = z_m + w, \quad z_N = z_m + y,
\]
where \( w = \sum_{i=1}^{n-m} a_{m+i} e_{i} \) and \( y = \sum_{i=1}^{N-m} a_{m+i} e_{i} \). By inequality (\ref{lembeq1}),
\[\begin{aligned}
\|z_n\| &\leq \|z_m\|+\|w\| \leq \|P(z_N)\|+\|w\| \leq C_P \|z_N\| + C_0 \|y\|  \\
&= C_P\|z_N\|+C_0 \|(I-P)(z_N)\| \leq \big(C_p+C_0(1+C_P)\big)\|z_N\|.
\end{aligned}
\]
Combining this with inequalities (\ref{lembeq4}), (\ref{lembeq5}), we simply take
 \[
 C = \max\left\{ C^2_1, C^2_1 C_p, C_p + C_0(1 + C_p) \right\}.
 \]
which yields inequality (\ref{lembeq3}).
\end{proof}

\begin{proposition}\label{lemx}
Let $s \in (1,2]$ and let $K \subseteq[0,1]$ be a compact set with $\mathcal{H}^{s-1}(K)> 0$. Suppose $f\in H_{s}(K)$.
Then there exists a sequence of compact sets $\{K_j\}_{j \geq 1}$ satisfying the following properties:
\begin{enumerate}
    \item $K_j \subseteq K$ and $\dim_H K_j = s-1, \ \forall j \geq 1$;
    \item $\mathrm{conv}(K_i) \cap \mathrm{conv}(K_j) = \emptyset, \ \forall i \neq j$;
    \item $f\Big|_{K \setminus \bigcup_{j \geq 1} K_j} \in H_s\Big(K \setminus \bigcup_{j \geq 1} K_j\Big)$.
\end{enumerate}

\end{proposition}
\begin{proof}
Let $x_0$ be a full-dimensional point of $f$, and take $r>0$ such that
\[
\mathcal{H}^{s-1}\big(K \setminus B(x_0, r)\big) > 0.
\]
We may then choose a sequence of compact sets $\{K_j\}_{j \geq 1}$ satisfying
\[
K_j \subseteq K \setminus B(x_0, r)\quad\text{and}\quad \dim_H K_j = s-1
\]
for any $j\geq 1$, together with
\[
\mathrm{conv}(K_i) \cap \mathrm{conv}(K_j) = \emptyset, \ \forall i\neq j.
\]
Since $K \cap B(x_0, r) \subseteq K \setminus \bigcup_{j \geq 1} K_j$, we conclude
\[
f\Big|_{K \setminus \bigcup_{j \geq 1} K_j} \in H_s\Big(K \setminus \bigcup_{j \geq 1} K_j\Big).
\]
\end{proof}

\medskip
We are now in a position to state and prove the main results of this section.
\begin{theorem}\label{thmn1} Let $s \in (1,2]$, let $\alpha$ be a cardinal number with $\alpha \geq \aleph_0$, and let $p=1$ or $p=2$. Then $H_s[0,1]$ is $(p, \mathfrak{c})$-spaceable, yet not $(\alpha, \beta)$-spaceable regardless of the cardinal number $\beta$.
\end{theorem}
\begin{proof}
From Corollary \ref{corn1}, it suffices to prove that $H_s[0,1]$ is $(p, \mathfrak{c})$-spaceable.
We only prove the case $p=2$; the case $p=1$ follows from Proposition \ref{lemx} via an analogous argument.

Assume that $f_1, f_2 \in H_s[0,1]$ are linearly independent and
\[
\operatorname{span}\{f_1, f_2\} \subseteq H_s[0,1] \cup \{0\}.
\]
By Proposition \ref{prop1}, there exists a sequence of compact sets $\{K_j\}_{j \geq 1}$ satisfying the following properties:

\begin{equation}\label{thm3eq1}
 \dim_H K_j = s - 1, \forall j \geq 1,
 \end{equation}
\begin{equation}\label{thm3nequ1}
   \operatorname{conv}(K_i) \cap \operatorname{conv}(K_j) = \emptyset, \forall i \neq j
\end{equation}
 and
\begin{equation}\label{thm3eq3}
    \left. \big(a f_1 + b f_2 \big)\right|_{[0,1] \setminus \bigcup_{j \geq 1} K_j} \in H_s\Big([0,1] \setminus \bigcup_{j \geq 1} K_j\Big)
\end{equation}
for any $(a,b) \in \mathbb{R}^2 \setminus \{(0,0)\}$.
Given an interval $I \subseteq [0,1]$, let $r_I$ and $l_I$ stand for its right and left endpoints, respectively. From now on, let $I_j = \operatorname{conv}(K_j)$ for each $j \in \mathbb{N}$.

By Lemma \ref{balk1}, for each $j \in \mathbb{N}$, we can choose a function $g_j \in C(K_j)$ satisfying
\[
\dim_H G_{g_j}(K_j) = s.
\]
By applying Lemma \ref{lemL}, we may further assume that
\[
g_j(r_{I_j}) = g_j(l_{I_j}) = 0.
\]
We extend $g_j$ linearly to the entire interval $[0,1]$, and denote this extension by $\widetilde{g}_j$, i.e.,
\[
\widetilde{g}_j(x) =
\begin{cases}
g_j(x), & x \in K_j, \\
0, & x \in [0,1] \setminus I_j, \\
\text{affine}, & \text{on each component of } I_j \setminus K_j.
\end{cases}
\]
By Lemma \ref{lemLW}, $\dim_H G_{\widetilde{g}_j}[0,1] = \dim_H G_{g_j}(K_j) = s$, so $\widetilde{g}_j\in H_{s}[0,1]$.
Accordingly, we obtain a sequence of continuous functions $\{\widetilde{g}_j:j \geq 1\}\subseteq H_{s}[0,1]$.
It follows from (\ref{thm3nequ1}) and the construction of $\widetilde{g}_j$ that $\{\operatorname{supp}(\widetilde{g}_j)\}_{j \geq 1}$ is pairwise disjoint, where $\operatorname{supp}(\widetilde{g}_j)$ denotes the support of the function $\widetilde{g}_j$, i.e.,
\[
\operatorname{supp}(\widetilde{g}_j) = \{ x \in [0,1] : \widetilde{g}_j(x) \neq 0 \}.
\]
Therefore, $\{\widetilde{g}_j:j\geq 1\}$ is linearly independent. Furthermore, $(\widetilde{g}_j)_{j\geq 1}$ forms a Schauder basic sequence, denote by $X_0 = \overline{\operatorname{span}}(\widetilde{g}_j)_{j \geq 1}$.
It is not difficult to verify that $\overline{f}_1, \overline{f}_2$ are linearly independent in the quotient space $C[0,1]/X_0$. Indeed, suppose for contradiction that $\overline{f}_1$, $\overline{f}_2$ are linearly independent in the quotient space; then there exists $(a,b) \neq (0,0)$ such that
\begin{equation}\label{thm3eq4}
a f_1 + b f_2 \in X_0.
\end{equation}
Observe that for any $h \in X_0$, there exists a sequence of scalars $(a_j)_{j \geq 1}$ such that
\[
h = \sum_{j=1}^\infty a_j \widetilde{g}_j.
\]
Then by the definitions of $\widetilde{g}_j, j\geq 1$, we have
\[
\left. h \right|_{[0,1] \setminus \bigcup_{j \geq 1} I_j} = 0 \quad \text{and} \quad \left. h \right|_{I_j} = a_j \widetilde{g}_j|_{I_j}, \ j \geq 1.
\]
It follows from the fact that $\widetilde{g}_j$ is affine on each component of $I_j \setminus K_j$ that
\[
\dim_H G_h(I_j \setminus K_j) \leq 1.
\]
Therefore we get
\[
\begin{aligned}
&\dim_H G_h\Big([0,1] \setminus \bigcup_{j \geq 1} K_j\Big)\\
=&\dim_H G_h\Big(\Big([0,1] \setminus \bigcup_{j \geq 1} I_j \Big)\ \cup \Big( \bigcup_{j \geq 1} I_j \setminus K_j\Big)\Big) \\
=&\max\Big\{ \dim_H G_h\Big([0,1] \setminus \bigcup_{j \geq 1} I_j\Big), \ \sup_{j \geq 1} \dim_H G_h(I_j \setminus K_j) \Big\}=1.
\end{aligned}
\]
Combining this estimation with (\ref{thm3eq3}) and (\ref{thm3eq4}), we reach a contradiction.
By applying Lemma \ref{lembasic}, the sequence $S = \{f_1, f_2, \widetilde{g}_1, \widetilde{g}_2, \dots\}$ forms a Schauder basic sequence.

Next, we prove that $\overline{\operatorname{span}} S \subseteq H_s[0,1] \cup \{0\}$.
Let $f \in \overline{\operatorname{span}} S \setminus \{0\}$, then there exists a nonzero sequence $(a_n)_{n \geq 1}$ of scalars, such that
\[
f=a_1f_1+a_2f_2+\sum_{n\geq 3}a_n\widetilde{g}_{n-2}.
\]
\noindent\textbf{First case:} $a_i \neq 0$ for some $i=1,2$.

 We rewrite $f$ by
    \[
    f= a_1f_1 + a_2f_2 + g,
    \]
    where $g = \sum_{n \geq 3} a_n \widetilde{g}_{n-2}$.
    Note that by (\ref{thm3eq1}),
   $
    \dim_H\left( \bigcup_{j \geq 1} K_j \right) = s - 1.
   $
    Hence, by inequality (\ref{equnew1}), we obtain
    \begin{equation}\label{thm3eq5}
    \dim_H G_f\Big( \bigcup_{j \geq 1} K_j \Big) \leq s.
    \end{equation}
    Since $g$ is affine on each component of $I_j \setminus K_j$ and
    \[
    \left. f \right|_{[0,1] \setminus \bigcup_{j \geq 1} I_j} = \left.\Big( a_1 f_1 + a_2 f_2\Big) \right|_{[0,1] \setminus \bigcup_{j \geq 1} I_j}.
    \]
    By applying Lemma \ref{lemL}, we obtain that 
    \[
    \begin{aligned}
    &\dim_H G_f\Big([0,1] \setminus \bigcup_{j \geq 1} K_j\Big)\\
    &= \dim_H G_f\Big(\Big([0,1] \setminus \bigcup_{j \geq 1} I_j \Big) \cup \Big( \bigcup_{j \geq 1} I_j \setminus K_j\Big)\Big) \\
    &= \max\Big\{ \dim_H G_{a_1 f_1 + a_2 f_2}\Big([0,1] \setminus \bigcup_{j \geq 1} I_j\Big), \ \sup_{j \geq 1} \dim_H G_f(I_j \setminus K_j) \Big\}\\
    &= \max\Big\{ \dim_H G_{a_1 f_1 + a_2 f_2}\Big([0,1] \setminus \bigcup_{j \geq 1} K_j\Big), \ \sup_{j \geq 1} \dim_H G_{a_1 f_1 + a_2 f_2}(I_j \setminus K_j) \Big\} \\
    &= \dim_H G_{a_1 f_1 + a_2 f_2}\Big([0,1] \setminus \bigcup_{j \geq 1} K_j\Big) = s.
    \end{aligned}
    \]
Combining this with inequality (\ref{thm3eq5}), we have $f \in H_s[0,1]$.

\medskip
\noindent\textbf{Second case:} $a_i = 0$ for any $i=1,2$.

 Since $f \neq 0$, there exists $i_0 \geq 3$ such that $a_{i_0} \neq 0$. It follows from
   $ \left. f \right|_{I_{i_0}} = \left. a_{i_0} \widetilde{g}_{i_0} \right|_{I_{i_0}}$
that
    \[
    \begin{aligned}
    \dim_H G_f([0,1]) &\geq \dim_H G_f(I_{i_0}) \\
    &= \dim_H G_{\widetilde{g}_{i_0}}(I_{i_0}) = \dim_H G_{g_{i_0}}(K_{i_0}) = s.
    \end{aligned}
    \]
    On the other hand,
    since $f$ is locally affine on $[0,1] \setminus \bigcup_{j \geq 1} K_j$ and
    \[
    \dim_H\Big( \bigcup_{j \geq 1} K_j \Big) = s - 1,
    \]
   we get that $\dim_H G_f([0,1]) \leq s$.
   Then we conclude that $\dim_H G_f([0,1]) = s$, i.e, $f \in H_s[0,1]$.

The proof is then complete.
\end{proof}

\begin{theorem}\label{thm2}
 Let $s \in (1,2]$ and $m,n \in \mathbb{N}$. Then $H_s[0,1]$ is $(n,m+n)$-lineable.
\end{theorem}
\medskip
\begin{proof}
Let $\{f_k\}_{k=1}^n \subseteq H_s[0,1]\setminus\{0\}$ be linearly independent and
\begin{equation}\label{thm4eq1}
\operatorname{span}\{f_k\}_{k=1}^n \subseteq H_s[0,1] \cup \{0\}.
\end{equation}
By Proposition \ref{prop2}, there exist compact sets $\{K_j\}_{j=1}^m$ such that

\begin{equation}\label{thm4eq2}
\dim_H K_j = s - 1, \forall j=1,2,\dots,m,
\end{equation}
\begin{equation}\label{thm4eq3}
\operatorname{conv}(K_i) \cap \operatorname{conv}(K_j) = \emptyset, \ \forall i \neq j,
\end{equation}
and
\begin{equation}\label{thm4eq4}
    \left. \sum_{i=1}^n a_i f_i \right|_{[0,1] \setminus \bigcup_{j=1}^m K_j} \in H_s\Big([0,1] \setminus \bigcup_{j=1}^m K_j\Big)
\end{equation}
for any $(a_1,\dots,a_n) \in \mathbb{R}^n \setminus \{(0,\dots,0)\}$.

By applying Lemma \ref{balk1}, for each $j \in \{1,2,\dots,m\}$, we can take a function $g_j \in C(K_j)$ such that
\[
\dim_H G_{g_j}(K_j) = s.
\]
Define $\widetilde{g}_j \in C[0,1]$ by
\[
\widetilde{g}_j(x) =
\begin{cases}
g_j(x), & x \in K_j, \\
\text{affine}, & \text{on each component of } [0,1] \setminus K_j.
\end{cases}
\]
We first prove that $\{f_1,\dots,f_n, \widetilde{g}_1,\dots,\widetilde{g}_m\}$ is linearly independent.
Assume that
\[
h := \sum_{i=1}^n a_i f_i + \sum_{j=1}^m b_j \widetilde{g}_j = 0.
\]
If the coefficients $\{a_i\}^n_{j=1}$ are not all zero, then
since $\widetilde{g}_j$ is affine on each component of $[0,1]\setminus K_j$, we have that
\begin{equation}\label{thm4eqn2}
\sum_{j=1}^m b_j \widetilde{g}_j\ \text{is also affine on each component of}\ [0,1]\setminus \bigcup_{j=1}^m K_j.
\end{equation}
Hence, we have, by (\ref{thm4eq4}) and Lemma \ref{lemL},
\[
\dim_H G_h\Big([0,1]\setminus \bigcup_{j=1}^m K_j\Big) = \dim_H G_{\sum_{i=1}^n a_i f_i}\Big([0,1]\setminus \bigcup_{j=1}^m K_j\Big) = s > 1.
\]
This contradicts the fact that $h=0$. So we have
\[
a_i = 0, \ i=1,2,\dots,n,
\]
which implies that
\[
h = \sum_{j=1}^m b_j \widetilde{g}_j = 0.
\]
Suppose that $b_{j_0} \neq 0$ for some $j_0 \in \{1,2,\dots,m\}$. Also since
$\widetilde{g}_j$ is affine on each component of $[0,1]\setminus K_j$, together with (\ref{thm4eq3}),
\[
\sum_{j\neq j_0} b_j \widetilde{g}_j \text{ is affine on } \operatorname{conv}(K_{j_0}).
\]
Therefore,
\[
\begin{aligned}
\dim_H G_h[0,1] &\geq \dim_H G_h\left(K_{j_0}\right) = \dim_H G_{\widetilde{g}_{j_0}}\left(K_{j_0}\right) \\
&= \dim_H G_{g_{j_0}}(K_{j_0}) = s.
\end{aligned}
\]
This contradicts $h=0$. Hence $b_j=0, \forall j$. This proves that
$\{f_1,\dots,f_n, \widetilde{g}_1,\dots,\widetilde{g}_m\}$ is linearly independent.

Finally, we prove that $\operatorname{span}\{f_1,\dots,f_n, \widetilde{g}_1,\dots,\widetilde{g}_m\}\subseteq H_{s}[0,1]\cup\{0\}$.
Let
\[f\in\operatorname{span}\{f_1,\dots,f_n, \widetilde{g}_1,\dots,\widetilde{g}_m\}\setminus\{0\}.
\]
Without loss of generality, we assume that $f=\sum_{i=1}^n a_i f_i + \sum_{j=1}^m b_j \widetilde{g}_j $.
From the preceding proof of linear independence, we clearly
conclude that
  \begin{equation}\label{thm4eq5}
\dim_H G_{\sum a_if_i+\sum b_j\tilde{g}_j}([0,1])\geq s
\end{equation}
for any $(a_1,\dots,a_n,b_1,\dots,b_m)\neq (0,\dots,0)$.
It follows from Lemma \ref{lemL} and (\ref{thm4eq1}), (\ref{thm4eqn2}) that
  \begin{equation}\label{thm4eq6}
\begin{aligned}
&\dim_H G_{\sum a_if_i+\sum b_j\tilde{g}_j}\Big([0,1]\setminus \bigcup_{j=1}^m K_j\Big) \\
=&\ \dim_H G_{\sum a_if_i}\Big([0,1]\setminus \bigcup_{j=1}^m K_j\Big)\leq s.
\end{aligned}
\end{equation}
Also by (\ref{equnew1}) and (\ref{thm4eq2}), we have $\dim_H \bigcup_{j=1}^m K_j = s-1$ and
\[
\dim_H G_{\sum a_if_i+\sum b_j\tilde{g}_j}\Big(\bigcup_{j=1}^m K_j\Big)\leq s.
\]
Combining this with (\ref{thm4eq5}) and (\ref{thm4eq6}), we conclude that every nonzero element

 \[
 f\in \operatorname{span}\{f_1,\dots,f_n,\tilde{g}_1,\dots,\tilde{g}_m\}\setminus\{0\}
 \]
 belongs to $H_s[0,1]$.

This completes the proof.
\end{proof}
\begin{remark} We point out that although our arguments are carried out for the subset $H_s([0,1])$ of space $C[0,1]$, Theorems \ref{thmn1} and \ref{thm2} established in this section can be easily reformulated for the space $C(K)$, where $K$ is a compact subset of $[0,1]$ with positive Hausdorff measure and $s\in(1,{\dim}_HK+1]$.
\end{remark}
The following two problems remain open.
\begin{question}Is the set $H_s[0,1]$ $(p,\mathfrak{c})$-spaceable for integers $3\leq p<\aleph_0$?
\end{question}
\begin{question}Is the set $\underline{B}_s[0,1]$ $(\alpha,\beta)$-lineable/spaceable for any two cardinal numbers $1<\alpha<\beta$?
\end{question}
\section{$(\alpha,\mathfrak{c})$-spaceability of functions for Upper box dimension}

In this section, we consider the spaceability of $\overline{B}_{s}[0,1]$. For this purpose, we first establish the following proposition.
\begin{proposition}\label{sec3prop1}
Let $s\in(1,2]$, and let $\{f_i\}_{i=1}^n\subseteq \overline{B}_s[0,1]$ be linearly independent. If
$\operatorname{span}\{f_i\}_{i=1}^n \subseteq \overline{B}_s[0,1]\cup\{0\},$ then there exists a nonempty open interval $I\subset[0,1]$ such that
\[
\left.\sum_{i=1}^n a_i f_i\right|_{[0,1]\setminus I} \in \overline{B}_s([0,1]\setminus I)
\]
for any $(a_1,a_2,\dots,a_n)\in\mathbb{R}^n\setminus\{(0,\dots,0)\}$.
\end{proposition}
\begin{proof}
We proceed by induction on $n$.

When $n=1$, by the finite stability of upper box dimension, we can take a nonempty open interval $I$ such that
\[
\overline{\dim}_B G_{f_1}([0,1]\setminus I)=s.
\]
Then interval $I$ satisfies the required property of this proposition.

Assume that the conclusion holds when $n=k$. Let $\{f_i:0\leq i\leq k\}$ be linearly independent subset of $\overline{B}_s[0,1]$ with
\[
\operatorname{span}\{f_i\}_{i=0}^k \subseteq \overline{B}_s[0,1]\cup\{0\}.
\]
It follows from the induction hypothesis that there exists an open
interval $I_1 \subseteq [0,1]$ such that
\[
\left.\sum_{i=1}^k a_i f_i\right|_{[0,1]\setminus I_1} \in \overline{B}_s([0,1]\setminus I_1),\ \forall\ (a_1,\dots,a_k) \neq (0,\dots,0).
\]
Without loss of generality, we assume that
\[
\overline{\dim}_B G_{f_0}([0,1]\setminus I_1) = s
\]
(if not, we may pass to a suitable subinterval of $I_1$).

\medskip
\textbf{a1)} Suppose
\[
\overline{\dim}_B G_{f_0 + \sum_{i=1}^k a_i f_i}([0,1]\setminus I_1) = s
\]
for any $(a_1,\dots,a_k) \in \mathbb{R}^k\setminus\{(0,\dots,0)\}$.

Then for any $(a_0,a_1,\dots,a_k) \in \mathbb{R}^{k+1}\setminus\{(0,\dots,0)\}$
(we may assume that $a_0 \neq 0$, otherwise, by the induction hypothesis, the following equality holds), we have
\[
\overline{\dim}_B G_{\sum_{i=0}^k a_i f_i}([0,1]\setminus I_1) = \overline{\dim}_B G_{a_0\left(f_0 + \sum_{i=1}^k a_0^{-1}a_i f_i\right)}([0,1]\setminus I_1) = s.
\]
Therefore, it suffices to take $I=I_1$ in this case.

\medskip
\textbf{a2)} Suppose there exists $(a_1^{(0)}, a_2^{(0)},\dots, a_k^{(0)}) \in \mathbb{R}^k\setminus\{(0,\dots,0)\}$ such that
\[
\overline{\dim}_B G_{f_0+\sum_{i=1}^k a_i^{(0)} f_i}([0,1]\setminus I_1) = s_0 < s.
\]

Then by the induction hypothesis and Lemma \ref{lemfalc},  we obtain
\begin{equation}\label{prop3eq1}
\begin{aligned}
&\overline{\dim}_B G_{f_0+\sum_{i=1}^k a_i f_i}([0,1]\setminus I_1) \\
&= \overline{\dim}_B G_{f_0+\sum_{i=1}^k a_i^{(0)} f_i+\sum_{i=1}^k (a_i-a_i^{(0)})f_i}([0,1]\setminus I_1) \\
&= s
\end{aligned}
\end{equation}
for any $(a_1,\dots,a_k)\neq (a_1^{(0)},\dots,a_k^{(0)})$.
Since
\[
\overline{\dim}_B G_{f_0+\sum_{i=1}^k a_i^{(0)} f_i}([0,1]) = s,
\]
we can take a nonempty open interval $I_2 \subseteq I_1$ such that
\begin{equation}\label{prop3eqn2}
\overline{\dim}_B G_{f_0+\sum_{i=1}^k a_i^{(0)} f_i}([0,1]\setminus I_2) = s.
\end{equation}
Then it follows from equalities (\ref{prop3eq1}), (\ref{prop3eqn2}) that
\[
\overline{\dim}_B G_{\sum_{i=0}^k a_i f_i}([0,1]\setminus I_2) = s
\]
for any $(a_0,a_1,\dots,a_k)\neq(0,\dots,0)$.
It suffices to take $I=I_2$. This completes the proof by induction.
\end{proof}

\begin{theorem}\label{thm3}
Let $s\in(1,2]$ and let $\alpha$ be a cardinal number. Then $\overline{B}_s[0,1]$ is $(\alpha,\mathfrak{c})$-spaceable if and only if $\alpha<\aleph_0$.
\end{theorem}
\medskip

\begin{proof} According to Corollary \ref{corn1}, we only need to prove that $\overline{B}_s[0,1]$ is $(n,\mathfrak{c})$-spaceable for each $n\in\mathbb{N}$.
Suppose that $\{f_i\}_{i=1}^n \subseteq \overline{B}_s[0,1]$ is linearly independent and
\[
\operatorname{span}\{f_i\}_{i=1}^n \subseteq \overline{B}_s[0,1]\cup\{0\}.
\]
In view of Proposition \ref{sec3prop1}, there exists a nonempty open interval $I\subseteq[0,1]$ such that
\begin{equation}\label{sec3eqn1}
\left.\sum_{i=1}^n a_i f_i\right|_{[0,1]\setminus I} \in \overline{B}_s([0,1]\setminus I)
\end{equation}
for any $(a_1,\dots,a_n)\neq(0,\dots,0)$. Take a sequence of compact subsets $\{E_k\}_{k\geq1}$ inside the interval $I$ satisfying the following properties:
\begin{enumerate}
    \item $E_k$ contains on isolated points for all $k\geq1$;
    \item $\overline{\dim}_B E_k = s-1$, $\forall k\geq1$;
    \item set $E = \overline{\bigcup_{k\geq1} E_k}$, then $\overline{\dim}_B E = s-1$;
    \item $\operatorname{conv}(E_i)\cap\operatorname{conv}(E_j) = \emptyset$, $\forall i\neq j$.
\end{enumerate}
Here, $\overline{A}$ stands for the closure of a set $A$.
In fact, such a set $E$ can be constructed as a ternary Cantor set contained in $I$ with contraction ratio $\mu = 2^{\frac{1}{1-s}}$, which satisfies $\dim_HE=\overline{\dim}_B E = s-1$.
Let $E_k$ be the pairwise disjoint fundamental components of the Cantor set $E$ with $E = \overline{\bigcup_{k=1}^\infty E_k}$.
For illustration, the Cantor set $E$ splits into a left part $F_L$ and right part $F_R$, then we let $E_1 = F_L$.
Similarly, the right part $F_R$ is a Cantor set, furthermore, we next let $E_2$ be the left part of the Cantor set $F_R$.
Repeating this recursive procedure yields the sequence $\{E_k\}_{k\geq1}$, and consequently $E =\overline{\bigcup_{k\geq1} E_k}=\bigcup_{k\geq1} E_k\cup\{r_{\operatorname{conv}(E)}\}$,
where $r_{\operatorname{conv}(E)}$ denotes the right endpoint of $\operatorname{conv}(E)$.

By applying Lemma \ref{balk2}, we take $h_k\in C(E_k)$, for each $k\in\mathbb{N}$ such that
\[
\overline{\dim}_B G_{h_k}(E_k) = s.
\]
By Lemma \ref{lemx2}, we may further assume that $h_k$ vanishes at the two endpoints of interval $\operatorname{conv}(E_k)$. We define $\tilde{h}_k$ by
\[
\tilde{h}_k(x) =
\begin{cases}
h_k(x), & x\in E_k, \\
\text{affine}, & \text{on each component of } \operatorname{conv}(E_k)\setminus E_k, \\
0, & x\in [0,1]\setminus\operatorname{conv}(E_k).
\end{cases}
\]
It follows from  Lemma \ref{lemLL} and the above construction that
\begin{equation}\label{sec3eq3}
\overline{\dim}_B G_{\tilde{h}_k}([0,1]) = s, \quad \forall k\in\mathbb{N}.
\end{equation}
Moreover, the family $\{\tilde{h}_k\}_{k\geq 1}$ has pairwise disjoint supports,
hence $(\tilde{h}_k)_{k\geq1}$ forms a Schauder basic sequence.

 Let $X_0 = \overline{\operatorname{span}}(\tilde{h}_k)_{k\geq1}$.
We now show that $\bar{f}_1,\dots,\bar{f}_n$ are linearly independent in $C[0,1]/X_0$.
Suppose that
\[
\sum_{k=1}^n a_k \bar{f}_k = \bar{0},
\]
which is equivalent to
\[
\sum_{k=1}^n a_k f_k \in X_0.
\]
Accordingly, by the definition of $\{\tilde{h}_k\}_{k\geq 1}$, we obtian
\[
\left. \sum_{k=1}^n a_k f_k \right|_{[0,1]\setminus I} = 0.
\]
This immediately yields $\overline{\dim}_B G_{\sum_{k=1}^n a_k f_k}([0,1]\setminus I) = 1$. Together with (\ref{sec3eqn1}), we deduce that
\[
(a_1,\dots,a_n) = (0,\dots,0).
\]
This confirms the linear independence of $\bar{f}_1,\dots,\bar{f}_n$ in  the quotient space $C[0,1]/X_0$.
By virtue of Lemma \ref{lembasic},
\[
\{f_1,\dots,f_n,\tilde{h}_1,\tilde{h}_2,\dots\}
\]
is a Schauder basic sequence.
We then reindex this sequence as
\[
 \{e_i\}_{i\geq 1}=\{f_1,\dots,f_n,\tilde{h}_1,\tilde{h}_2,\dots\}.
\]
We now proceed to verify that
\[
\overline{\operatorname{span}}(e_i)_{i\geq 1} \subseteq \overline{B}_s[0,1] \cup \{0\}.
\]
Let $f \in \overline{\operatorname{span}}(e_i)_{i\geq 1} \setminus \{0\}$. Then there exists a nonzero sequence
$(a_i)_{i\geq 1}$ of scalars such that
\[
f = \sum_{i=1}^\infty a_i e_i = \sum_{k=1}^n a_k f_k + \sum_{i=1}^\infty a_{n+i} \tilde{h}_i.
\]
On the one hand, from the construction of $\tilde{h}_k$, the function $\sum_{i=1}^\infty a_{n+i} \tilde{h}_i$
is affine on each component of $[0,1]\setminus E$. Therefore, by applying
Lemma \ref{lemLL},
\[
\overline{\dim}_B G_{\sum_{i=1}^\infty a_{n+i} \tilde{h}_i}([0,1]) \leq \overline{\dim}_B E + 1 = s.
\]
Combining this with the fact that $f_i \in \overline{B}_s[0,1]$, $i=1,2,\dots,n$ and Lemma \ref{lemfal-fr}, we have
\[
\overline{\dim}_B G_f([0,1]) \leq s.
\]
On the other hand,
If $(a_1,\dots,a_n) \neq (0,\dots,0)$, then
\[
\begin{aligned}
\overline{\dim}_B G_f([0,1]) &\geq \overline{\dim}_B G_f([0,1]\setminus I) \\
&= \overline{\dim}_B G_{\sum_{k=1}^n a_k f_k}([0,1]\setminus I) = s.
\end{aligned}
\]
If $(a_1,\dots,a_n) = (0,\dots,0)$, then there exists $i_0 \geq 1$ such that $a_{n+i_0} \neq 0$, and it follows from equality (\ref{sec3eq3}) and  $f|_{E_{i_0}} = a_{n+i_0} h_{i_0}$ that
\[
\overline{\dim}_B G_f([0,1]) \geq \overline{\dim}_B G_f(E_{i_0}) = \overline{\dim}_B G_{h_{i_0}}(E_{i_0}) = s.
\]
In all cases, we conclude $f\in\overline{B}_s[0,1]$.
\end{proof}

\medskip
At the end of this section, we briefly discuss the $(\alpha,\beta)$-lineability of the intersection $H_s[0,1]\cap B_s[0,1]$ with the previously established methods.
\begin{lemma}\label{lemn1}
Let $s\in(1,2]$, $n\in\mathbb{N}$ and let $\{f_i\}_{i=1}^n\subseteq H_s[0,1]\cap B_s[0,1]$ be linearly independent.
If
\begin{equation}\label{sec3eqn3}
\operatorname{span}\{f_i\}_{i=1}^n \subseteq \big(H_s[0,1]\cap B_s[0,1]\big)\cup\{0\},
\end{equation}
 then for any $m\in\mathbb{N}$, there exist compact sets $\{K_j\}_{j=1}^m$ satisfying the following properties:
\begin{enumerate}
\item $K_j$ contains no isolated points for each $j$;
\item $\dim_H K_j = \dim_B K_j = s-1$ for each $j$;
\item $\operatorname{conv}(K_i)\cap \operatorname{conv}(K_j)=\emptyset,\ \forall\,i\neq j$;
\item For any $(a_1,\dots,a_n)\in\mathbb{R}^n\setminus\{(0,\dots,0)\}$,
\[
\left.\sum_{i=1}^n a_i f_i\right|_{[0,1]\setminus\bigcup_{j=1}^m K_j}
\in H_s\Big([0,1]\setminus\bigcup_{j=1}^m K_j\Big)\bigcap B_s\Big([0,1]\setminus\bigcup_{j=1}^m K_j\Big).
\]
\end{enumerate}
\end{lemma}

\begin{proof}
By Proposition 4.1, there exists a nonempty interval $I\subseteq[0,1]$ such that
\begin{equation}\label{sec3eqn2}
\left.\sum_{i=1}^n a_i f_i\right|_{[0,1]\setminus I}
\in \overline{B}_s\left([0,1]\setminus I\right)
\end{equation}
for any $(a_1,\dots,a_n)\in\mathbb{R}^n\setminus\{(0,\dots,0)\}$.
Let $K\subseteq I$ be a compact set with
\[
0<\mathcal{H}^{s-1}(K)<\infty \quad\text{and}\quad \dim_H K = \dim_B K = s-1.
\]
For each $j\in\{1,2,\dots,n\}$, we define $g_j$ and $A_j$ by
\[
g_j = f_j\big|_K \quad \text{and} \quad
A_j = \bigl\{x\in K: {\dim}_HG_{g_j}(B(x,r)\cap K)=s\ \text{for any}\ r>0\bigr\},
\]
respectively. We refer to the elements of $A_j$ as the full $s$-dimensional point of $g_j$. Note that under this new notation, full $s$-dimensional points of $g_j$ may not exist.
Let $\mathcal{P}_{\{g_j\}}$ be the set of all full $s$-dimensional points of functions belonging to $\operatorname{span}\{g\}^n_{j=1}\setminus\{0\}.$
 And then we shall consider the functions $g_j\in C(K),j=1,2,\dots,n$.
 Using an argument analogous to that in the proof of Proposition \ref{prop2}, where we replace full-dimensional points and $P_{\{f_j\}}$ with full $s$-dimensional points and $\mathcal{P}_{\{g_j\}}$ repectively, we can conclude that there exist $m$ compact subsets of $K$, denoted by $\{K_j\}^m_{j=1}$ such that
\begin{itemize}
\item $\dim_H K_j=s-1$, $\forall$ $j=1,2,\dots,m$;
\item $\operatorname{conv}(K_i)\cap \operatorname{conv}(K_j)=\emptyset$, $\forall$ $i\neq j$;
\item for any $(a_1,\dots,a_n)\in\mathbb{R}^n$,
\[\left.\sum_{i=1}^n a_i f_i\right|_{\bigcup_{j=1}^m K_j}=\left.\sum_{i=1}^n a_i g_i\right|_{\bigcup_{j=1}^m K_j}
\in H_{<s}\Big(\bigcup_{j=1}^m K_j\Big).\]
\end{itemize}
We can remove countably many points from $K_j$ without changing the above three properties, so we may assume $K_j$ contains no isolated points for each $j$ . Since ${\dim}_BK=s-1$, we have ${\dim}_BK_j= s-1$ for any $j$.
Combining these with (\ref{sec3eqn3}) and (\ref{sec3eqn2}), we complete the proof of the lemma.
\end{proof}

Combining the proof techniques from Theorem \ref{thm2}, the dimensional arguments in Theorem \ref{thm3} and Lemma \ref{lemn1}, we derive the following theorem. As its proof follows by an almost identical argument to Theorem \ref{thm2}, we omit the detailed proof.
\begin{theorem}
Let $s\in(1,2]$ and $m, n\in\mathbb{N}$. Then $H_s[0,1]\cap B_s[0,1]$ is $(n,n+m)$-lineable.
\end{theorem}

\end{document}